\documentclass[11pt]{amsart}

\usepackage[T1]{fontenc}
\usepackage{lmodern}
\usepackage{microtype}
\usepackage{amsmath,amssymb,amsthm,mathtools,mathrsfs}
\usepackage{booktabs}
\usepackage{enumitem}
\usepackage{xcolor}
\usepackage{geometry}
\usepackage{aliascnt}
\usepackage{hyperref}
\usepackage[nameinlink,capitalise,noabbrev]{cleveref}
\usepackage{array}

\hypersetup{
  colorlinks=true,
  linkcolor=blue!55!black,
  citecolor=blue!55!black,
  urlcolor=blue!55!black,
  pdfauthor={Gerald H\"ohn},
  pdftitle={The Hall--Janko Simple Group and the Icosian Leech Lattice}
}

\newtheorem{theorem}{Theorem}[section]
\newaliascnt{proposition}{theorem}
\newtheorem{proposition}[proposition]{Proposition}
\aliascntresetthe{proposition}
\newaliascnt{lemma}{theorem}
\newtheorem{lemma}[lemma]{Lemma}
\aliascntresetthe{lemma}
\newaliascnt{corollary}{theorem}
\newtheorem{corollary}[corollary]{Corollary}
\aliascntresetthe{corollary}
\newaliascnt{conjecture}{theorem}

\aliascntresetthe{conjecture}
\theoremstyle{definition}
\newaliascnt{definition}{theorem}
\newtheorem{definition}[definition]{Definition}
\aliascntresetthe{definition}
\newaliascnt{construction}{theorem}

\aliascntresetthe{construction}
\newaliascnt{problem}{theorem}
\newtheorem{problem}[problem]{Problem}
\aliascntresetthe{problem}
\theoremstyle{remark}
\newaliascnt{remark}{theorem}
\newtheorem{remark}[remark]{Remark}
\aliascntresetthe{remark}

\newcommand{\Z}{\mathbb Z}
\newcommand{\Q}{\mathbb Q}
\newcommand{\R}{\mathbb R}
\newcommand{\C}{\mathbb C}
\newcommand{\Hh}{\mathbb H}
\newcommand{\F}{\mathbb F}
\newcommand{\A}{\mathbb A}
\newcommand{\cO}{\mathcal O}
\newcommand{\cI}{\mathcal I}
\newcommand{\cG}{\mathcal G}
\newcommand{\cH}{\mathcal H}
\newcommand{\Aut}{\operatorname{Aut}}
\newcommand{\Tr}{\operatorname{Tr}}
\newcommand{\tr}{\operatorname{tr}}
\newcommand{\trd}{\operatorname{trd}}
\newcommand{\nrd}{\operatorname{nrd}}
\newcommand{\Mass}{\operatorname{Mass}}
\newcommand{\Min}{\operatorname{Min}}
\newcommand{\Span}{\operatorname{span}}
\newcommand{\Herm}{\operatorname{Herm}}
\newcommand{\SL}{\operatorname{SL}}
\newcommand{\Sp}{\operatorname{Sp}}
\newcommand{\Fr}{\operatorname{Fr}}
\newcommand{\HJ}{L_{\mathrm T}}
\newcommand{\Ls}{L_{\mathrm{spl}}}
\newcommand{\dK}{\mathfrak d_K}
\newcommand{\PP}{\mathbb P}
\newcommand{\eK}{\mathbf e_K}

\title[Hall--Janko and the icosian Leech lattice]{The Hall--Janko Simple Group\\and the Icosian Leech Lattice}
\author{Gerald H\"ohn}
\address{Department of Mathematics, Kansas State University, Manhattan, Kansas 66506, USA}
\email{gerald@monstrous-moonshine.de}
\date{September 2026}

\begin{document}

\begin{abstract}
We give a unified arithmetic and geometric treatment of the Hall--Janko
simple group and the icosian Leech lattice.  Let~$\cI$ be the standard
icosian maximal order over $K=\Q(\sqrt5)$.  Following the
mass-formula and Venkov framework for the Niemeier and Leech lattices,
we classify positive definite unimodular Hermitian~$\cI$-lattices of
quaternionic rank~$3$.

Hashimoto's mass formula, the Siegel--Weil average, and Gundlach's
structure theorem determine the scalar theta series of every nonsplit
class.  A two-place harmonic separation of the Leech minimal shell
proves that its three Hermitian components are quaternionic projective~$5$-designs.  For the norm-$2$ shell, integrality and projective moments
recover the complete angle scheme and its~$525$ orthogonal frames.
The finite code of one frame then saturates the residual mass, giving
exactly the split and Tits classes and an independent computation of
the order of $\Aut_{\cI}(L_{\mathrm T})\cong2.J_2$.

Finally, a quaternionic line system with the prescribed five projective
inner products has at most~$315$ points, with equality precisely for
the Hall--Janko configuration up to~$P\Sp(3)$.  Equality in Hoggar's
bound supplies the design property; Cohen--Tits uniqueness, golden-angle
fission, and the centered cubic moment give geometric rigidity.
\end{abstract}

\maketitle

\section{Introduction}\label{sec:introduction}

\subsection*{From the Hall--Janko group to the icosian Leech lattice}

The purpose of this paper is to give a unified arithmetic and geometric
treatment of the Hall--Janko simple group and the icosian Leech lattice,
analogous to the standard treatment of the Conway groups through the
Leech lattice.

Janko predicted the simple group now denoted~$J_2$ from its involution
centralizer, and Hall and Wales constructed the group of order
\[
 |J_2|=604\,800=2^7 3^3 5^2 7
\]
as a rank-$3$ permutation group on~$100$ points \cite{Janko,HallWales}.
Wales embedded it in~$G_2(4)$ and gave explicit generators \cite{Wales}.
The Schur multiplier has order~$2$, and Lindsey found the faithful
degree-$6$ representation of the double cover~$2.J_2$, first announced
in 1968 and developed explicitly in 1970
\cite{LindseyAnnouncement,Lindsey}.  His later work produced a different~$24$-dimensional integral lattice invariant under~$2.J_2$
\cite{LindseyLattice}.  After splitting the quaternion algebra, the
rank-$3$ representation considered below becomes Lindsey's symplectic
representation,
\[
 2.J_2<\Sp_6(\C)<F_4(\C)<E_6(\C),
\]
and its occurrence in the exceptional complex groups is recorded in the
finite-subgroup analysis of Cohen--Wales \cite{CohenWales}.

Tits's construction gives a more arithmetic realization.  Let
\[
 K=\Q(\sqrt5),
 \qquad
 \tau=\frac{1+\sqrt5}{2},
 \qquad
 D=\left(\frac{-1,\,-1}{K}\right),
\]
where $D$ is the Hamilton quaternion algebra over~$K$, and let
$\cI\subset D$ be an icosian maximal order over~$\cO_K=\Z[\tau]$.
Tits constructs an
explicit self-dual Hermitian~$\cI$-lattice~$L_{\mathrm T}$ of quaternionic
rank~$3$ whose associated trace lattice is the Leech lattice
\cite{Tits}.  Cohen's classification places the corresponding finite
group among the quaternionic reflection groups \cite{CohenReflections},
while Wilson identifies and analyzes the~$2.J_2$ symmetry and its
distinguished quaternionic configurations \cite{Wilson,WilsonBook}.

The results assembled here are largely known in these separate forms.
We derive the shell geometry, the frame count, and the relevant
automorphism-group order from the arithmetic of the genus before using
the finite-group identification.

\subsection*{The arithmetic genus and its theta series}

The arithmetic classification has been approached from several
directions.  Costello--Hsia study the related rank-$12$ orthogonal setting
over~$\cO_K$, whereas Coulangeon treats unimodular quaternionic lattices
by neighbor methods and obtains the present rank-$3$ class-number result
\cite{CostelloHsia,Coulangeon94,Coulangeon95}.  Nebe places~$2.J_2$ in the
general theory of finite quaternionic matrix groups and, later, in the
theory of golden lattices \cite{NebeQuaternionic,NebeGolden}; Kirschmer
records the same quaternionic genus in his systematic small-class-number
classification \cite{Kirschmer}.

Our proof instead combines the Hashimoto--Shimura mass formula, the
Siegel--Weil average, Gundlach's ring of Hilbert modular forms, and the
finite~$\F_4$-shadow of an orthogonal frame.  The frame count required by
the mass argument is derived intrinsically from the projective shell
rather than imported from coordinates.  This gives the following
classification in the normalization used throughout the paper.

\begin{theorem}[mass-assisted~$C_3$ classification]\label{thm:main}
Let~$\cG_3$ be the genus of positive definite unimodular right~$\cI$-lattices of rank~$3$ in the standard Hermitian space.
Here~$\Aut_{\cI}(L)$ denotes the group of~$\cI$-linear isometries of~$L$.
Then
\[
 \cG_3=\{[\Ls],\,[\HJ]\},
 \qquad
 \Ls=\cI^3,
\]
and
\[
 \Aut_{\cI}(\Ls)\cong\SL_2(5)^3\rtimes S_3,
 \qquad
 \Aut_{\cI}(\HJ)\cong2.J_2.
\]
Their orders are~$10\,368\,000$ and~$1\,209\,600$, respectively.
\end{theorem}

The scalar Hilbert theta series separates the split class from every
other class before the class number is known.  The first Fourier
coefficient saturates the Eisenstein average, forcing every nonsplit
class to have no vectors of Hermitian norm~$1$; such a class is called
\emph{extremal}.  All nonsplit classes have the same theta series.
Write~$E_{k,\,K}$ for the parallel-weight-$k$ Hilbert Eisenstein series
normalized to have constant term~$1$.

\begin{theorem}[scalar Hilbert theta series]\label{thm:scalar-intro}
Every nonsplit class in~$\cG_3$ is extremal and has scalar Hilbert theta
series
\[
 \Theta_{\mathrm{ext}}
 =\frac{67E_{6,\,K}-7E_{2,\,K}^{3}}{60}.
\]
Writing~$r_{\mathrm{ext}}(a)$ for the number of vectors of Hermitian
norm~$a$, its first nonzero coefficients are
\[
 r_{\mathrm{ext}}(2)=37\,800,
 \qquad
 r_{\mathrm{ext}}(2+\tau)=120\,960,
 \qquad
 r_{\mathrm{ext}}(2\tau^2)=37\,800.
\]
The corresponding projective shells are sets of~$315$, $1008$, and~$315$ points, respectively.
\end{theorem}

This follows from \cref{prop:coefficient-saturation},
\eqref{eq:extremal-theta}, and the shell calculation in
\cref{sec:extremal-shell-counts}.

\subsection*{Projective designs and the Hall--Janko near octagon}

The geometry of the minimal shell connects several older constructions.
Hoggar developed linear-programming bounds and design theory for
quaternionic projective spaces
\cite{HoggarBounds,HoggarDesigns,HoggarParameters,HoggarTight}.  His
Example~3.9 is a system of~$315$ quaternionic lines whose projective inner
products---the squared absolute Hermitian inner products of unit representatives---are
\begin{equation}\label{eq:hoggar-angle-set-intro}
 \left\{0,\,\frac12,\,\frac14,\,
 \frac{3+\sqrt5}{8},\,
 \frac{3-\sqrt5}{8}\right\},
\end{equation}
and the annihilator
\begin{equation}\label{eq:hoggar-polynomial-intro}
 F(t)=\frac1{15}t(2t-1)(4t-1)(16t^2-12t+1)
\end{equation}
gives the upper bound~$|X|\le315$ for line systems with this prescribed
angle set \cite[Example~3.9]{HoggarBounds}.  Equality forces the
projective~$5$-design property, as we verify in
\cref{lem:hoggar-bound-equality}.  The Hall--Janko system is therefore
maximal for its angles.  It is not a tight projective~$5$-design: a
tight projective~$5$-design is a three-distance set, whereas the present
configuration has five off-diagonal projective inner products
\cite{HoggarTight}.

Independently, Cohen introduced the~$315$-point near octagon associated
with~$J_2$ \cite{CohenNearOctagon,CohenGeometry}, and Cohen--Tits proved
that every regular near octagon with parameters~$(2,\,4;0,\,3)$ is
isomorphic to it \cite{CohenTits}.  Tits and Wilson recorded the five
projective inner products in explicit icosian coordinates.  Here these
features are recovered from the arithmetic shell itself.

The trace lattice of an extremal class is the Leech lattice.  Its
harmonic theta identities, combined with a two-place radial separation,
give a rank-$3$ icosian analogue of Venkov's design theorem.

\begin{theorem}[rank-$3$ icosian Venkov theorem]\label{thm:venkov-intro}
Let~$L\in\cG_3$ be extremal.  The projectivized Hermitian shells
of norms
\[
 2,\qquad 2+\tau,\qquad 2\tau^2
\]
are quaternionic projective~$5$-designs in~$\Hh\PP^2$, with~$315$, $1008$, and~$315$ points, respectively.  Multiplication by~$\tau$ identifies the first and third projectively.
\end{theorem}

The design assertion is proved in \cref{thm:rank3-venkov};
\cref{lem:primitive-shell-lines} supplies the cardinalities and
the projective identification.

For the norm-$2$ shell, icosian integrality, extremality, and reduction
modulo~$2$ first restrict the projective inner products to the five
values in \eqref{eq:hoggar-angle-set-intro}.  An arithmetic primitivity lemma excludes repeated lines; the
one-point projective moments determine the valencies, and the mixed
moments determine all intersection numbers.  No point-stabilizer
suborbit calculation is used.

\begin{theorem}[intrinsic norm-$2$ shell geometry]\label{thm:design-intro}
Let~$L\in\cG_3$ be extremal, and let~$X_L$ be its projectivized norm-$2$
shell.  Then~$X_L$ is a set of~$315$ points.  The projective inner
products between distinct points and their multiplicities relative to
every point are
\[
\begin{array}{c|ccccc}
 t&0&\frac12&\frac14&\frac{3+\sqrt5}{8}&\frac{3-\sqrt5}{8}\\ \hline
 \text{multiplicity}&10&80&160&32&32.
\end{array}
\]
The projective~$5$-design identities force all intersection numbers of
the resulting symmetric five-class association scheme.  Fusing the two
golden-angle relations gives the distance scheme of the orthogonality
graph, with intersection array
\[
 \{10,\,8,\,8,\,2;1,\,1,\,4,\,5\}.
\]
Its triangles are the~$525$ orthogonal quaternionic frames, and with
these triangles as lines~$X_L$ is a regular near octagon with parameters~$(2,\,4;0,\,3)$.
\end{theorem}

The proof is given by \cref{prop:extremal-shell-angles},
\cref{lem:norm2-multiplicity-free}, \cref{thm:angle-association-scheme},
\cref{prop:distance-angle}, and \cref{thm:near-octagon}.

This theorem supplies the~$525$ frames used in the arithmetic
classification.  Conversely, after the mass argument identifies the
extremal class with Tits's lattice, Cohen--Tits uniqueness identifies the
fused incidence geometry with the Hall--Janko near octagon.  A connected
bipartition argument then shows that its distance-$4$ relation has, up to
simultaneous exchange, a unique fission with the intersection parameters
forced by the projective moments.

\subsection*{From combinatorial uniqueness to geometric rigidity}

Uniqueness of the near octagon alone does not determine a realization in~$\Hh\PP^2$, because it fuses the two golden-angle relations.  The
remaining geometric step is supplied by the cubic structure of the
Jordan algebra~$\Herm_3(\Hh)$ of~$3\times3$ Hermitian quaternionic matrices.  The centered second moment reconstructs the Euclidean
metric on its trace-zero subspace~$\Herm_3(\Hh)^0$, while the centered third moment reconstructs
its invariant cubic form.  Polarization then recovers the Jordan product
and reduces the ambient orthogonal symmetry to~$P\Sp(3)$.

\begin{theorem}[quaternionic cubic rigidity]\label{thm:cubic-rigidity-intro}
Let~$Y$ and~$Y'$ be finite quaternionic projective~$3$-designs in~$\Hh\PP^2$,
and let
\[
 \phi:Y\longrightarrow Y'
\]
be a bijection preserving all projective inner products.  Then there is
a unique element
\[
 g\in P\Sp(3)=\Aut(\Herm_3(\Hh))
\]
such that $g(y)=\phi(y)$ for every~$y\in Y$.
\end{theorem}

The proof is given in \cref{thm:quaternionic-cubic-rigidity}.
The association-scheme and fission arguments supply the required
bijection preserving all projective inner products in the Hall--Janko case.

\begin{theorem}[geometric rigidity for the prescribed angles]\label{thm:geometric-rigidity-intro}
A subset of~$\Hh\PP^2$ whose off-diagonal projective inner products
belong to \eqref{eq:hoggar-angle-set-intro} has at most~$315$ points.
Equality holds if and only if it is~$P\Sp(3)$-equivalent to the
Hall--Janko configuration.
\end{theorem}

The proof is given in \cref{thm:HJ-prescribed-angle-uniqueness}.

The prescribed-angle hypothesis remains essential to what is proved
here; the design property is a consequence of equality, not an
additional hypothesis.  We do not claim conversely that cardinality~$315$ and the projective~$5$-design property alone force the five values
in \eqref{eq:hoggar-angle-set-intro}.

\subsection*{Inputs and contributions}

The external arithmetic inputs are the Hashimoto--Shimura mass formula,
the genus-one Siegel--Weil identity, and Gundlach's description of
level-one Hilbert modular forms over~$\Q(\sqrt5)$.  The geometric inputs
are the uniqueness and harmonic theta properties of the Leech lattice
and the Cohen--Tits uniqueness theorem for the regular near octagon.
Tits's explicit lattice supplies the concrete nonsplit class and, only
at the final identification step, the reflection group~$2.J_2$.
The automorphic route is deliberate: the genus average and harmonic
theta identities place the construction in the same framework as the
classical Niemeier--Leech and Venkov arguments.  Siegel--Weil is used in
\cref{prop:coefficient-saturation} and \eqref{eq:extremal-theta}; in this
small genus those conclusions also follow from rank-$2$ class number
one and Gundlach's theorem, as noted in \cref{rem:rank-two-shortcut}.

The arguments supplied here are the coefficient-saturation proof of
extremality, the two-place separation of the three Hermitian pieces of
the Leech minimal shell, the intrinsic norm descent for the five
projective inner products, the recovery of the complete angle scheme
from projective moments, the uniqueness of its golden fission, the
mass-assisted frame-code classification, and the cubic Jordan rigidity
theorem.  In particular, neither the~$525$-frame count nor the order~$1\,209\,600$ is imported from a prior finite-group identification.  The
name~$2.J_2$ is used only after the lattice classification and the
automorphism-group order have been obtained.

\Cref{sec:dictionary} fixes the arithmetic and adelic normalization.
\Cref{sec:scalar-theta} evaluates the mass and scalar theta series and
proves extremality.  \Cref{sec:rank3-venkov} passes through the Leech
trace lattice to the projective~$5$-design identities, and
\cref{sec:norm2-geometry} reconstructs the norm-$2$ shell geometry.
\Cref{sec:tits-lattice} uses the intrinsic frames and their Morita image
to saturate the residual mass, while \cref{sec:geometric-rigidity}
proves the combinatorial and geometric rigidity statements.

\section{The compact group of type \texorpdfstring{$C_3$}{C3} and its genus}\label{sec:dictionary}

\subsection{The field and the icosian order}

We label the real embeddings~$\sigma_1$ and~$\sigma_2$ by
$\sigma_1(\sqrt5)=\sqrt5$ and $\sigma_2(\sqrt5)=-\sqrt5$.
The nontrivial Galois automorphism is denoted by~$a\mapsto a'$.
Let
\[
 D=\left(\frac{-1,\,-1}{K}\right)
 =K\langle i,\,j\rangle/(i^2=j^2=-1,\,ij=-ji),
 \qquad k=ij.
\]
For $q=a+bi+cj+dk$ with~$a$, $b$, $c$, and~$d$ in~$K$, quaternionic
conjugation, reduced norm, and reduced trace mean
\[
 \overline q=a-bi-cj-dk,
 \qquad \nrd(q)=q\overline q,
 \qquad \trd(q)=q+\overline q.
\]
Both reduced maps take values in~$K$; in particular, $\trd(a)=2a$
for~$a\in K$.  For matrices and column vectors, $A^*=\overline A^{\,T}$
denotes conjugate transpose, while~$A^T$ denotes ordinary transpose.
The algebra~$D$ is the scalar extension to~$K$ of the Hamilton quaternion algebra over~$\Q$, which is ramified only at~$2$ and at infinity.  The prime~$2$ is inert in~$K$, so the local degree at the unique dyadic place is~$2$ and the invariant~$1/2$ becomes~$0$ after base change.  The invariants at all other finite places were already zero.  Thus~$D$ is ramified at the two real places and split at every finite place.  We use the standard icosian maximal order
\begin{equation}\label{eq:standard-icosian-order}
 \cI=\left\langle
  1,\,i,\,\frac{1+i+j+k}{2},\,
  \frac{1-\tau+\tau i+k}{2}
 \right\rangle_{\cO_K}.
\end{equation}
For $q=a+bi+cj+dk$, define
\begin{equation}\label{eq:icosian-galois-map}
 \widetilde q=a'+b'i+d'j+c'k,
 \qquad
 \gamma(q)=\overline{\widetilde q}.
\end{equation}
A direct check on~$i$ and~$j$ shows that~$q\mapsto\widetilde q$ is a semilinear anti-involution.  Write
\[
 e=\frac{1+i+j+k}{2},
 \qquad
 f=\frac{1-\tau+\tau i+k}{2}.
\]
Then
\[
 \widetilde 1=1,
 \quad \widetilde i=i,
 \quad \widetilde e=e,
 \quad \widetilde f=e-f,
\]
so $\widetilde{\cI}=\cI$.  Moreover,
\[
 \overline e=1-e,
 \qquad
 \overline f=(1-\tau)-f,
\]
so canonical quaternionic conjugation also preserves~$\cI$.  Thus~$\gamma$ is a semilinear algebra automorphism of~$D$ preserving~$\cI$, with
\begin{equation}\label{eq:icosian-galois-properties}
 \gamma(qr)=\gamma(q)\gamma(r),\qquad
 \gamma(\overline q)=\overline{\gamma(q)},\qquad
 \nrd(\gamma(q))=\nrd(q)'.
\end{equation}
The norm-one group of~$\cI$ is
\begin{equation}\label{eq:icosian-units}
 \cI^1=\{u\in\cI:\nrd(u)=1\}
 \cong\SL_2(5),
 \qquad
 |\cI^1|=120.
\end{equation}
At the prime above~$2$, reduction induces
\begin{equation}\label{eq:reduction-mod2}
 \cI/2\cI\cong M_2(\F_4),
 \qquad
 \cI^1/\{\pm1\}\xrightarrow{\sim}\SL_2(4),
\end{equation}
where reduced norm becomes determinant; equivalently, the reduction map~$\cI^1\to\SL_2(4)$ is surjective with kernel~$\{\pm1\}$.

All quaternionic modules are right modules, and Hermitian forms are
conjugate-linear in the first variable and linear in the second.
On~$D^3$, with standard coordinate vectors~$e_1$, $e_2$, and~$e_3$, use
the positive definite Hermitian form
\begin{equation}\label{eq:standard-hermitian}
 h(x,\,y)=\sum_{j=1}^3\overline{x_j}y_j.
\end{equation}
The Hermitian norm of~$x$ means~$h(x,\,x)$.  For a right~$\cI$-lattice
$L\subset D^3$, set
\[
 L^\#=\{x\in D^3:h(x,\,L)\subseteq\cI\}.
\]
We call~$L$ unimodular if $L=L^\#$.

\subsection{The algebraic group and the class set}

Let~$\mathbf G$ be the algebraic~$K$-group of~$D$-linear isometries of
$(D^3,\,h)$.  It is an inner~$K$-form of the simply connected group of type~$C_3$.  At the two real places and at every finite place,
\[
 \mathbf G(K_\sigma)\cong\Sp(3),
 \qquad
 \mathbf G(K_{\mathfrak p})\cong\Sp_6(K_{\mathfrak p}).
\]
Thus~$\mathbf G$ is compact at infinity and unramified at every finite prime.  Although~$\mathbf G$ is simply connected, strong approximation with respect to the infinite places does not collapse the class set below, because~$\mathbf G(K\otimes_\Q\R)$ is compact.

Let~$\cG_3$ be the genus of unimodular rank-$3$ $\cI$-lattices.  Ad\`elically,
\begin{equation}\label{eq:adelic-class-set}
 \cG_3\cong
 \mathbf G(K)\backslash\mathbf G(\A_{K,\,f})/
 \prod_{\mathfrak p}\mathbf G(\cO_{K,\,\mathfrak p}).
\end{equation}
Here~$\A_{K,\,f}$ is the ring of finite adeles and
$\mathbf G(\cO_{K,\,\mathfrak p})$ is the stabilizer of the local
lattice~$\cI_{\mathfrak p}^3$.
Its mass is
\begin{equation}\label{eq:mass-def}
 \Mass(\cG_3)
 =\sum_{[L]\in\cG_3}\frac1{|\Aut_{\cI}(L)|}.
\end{equation}

\begin{remark}[The quaternionic genus and its stabilizers]\label{rem:integral-categories}
The class set~$\cG_3$ is a genus of rank-$3$ Hermitian lattices over~$\cI$, and the groups occurring in \eqref{eq:mass-def} are their~$\cI$-linear isometry groups.  Forgetting the quaternionic action places such a lattice in the rank-$12$ orthogonal setting over~$\cO_K$.  The even unimodular determinant-one orthogonal genus classified by Costello--Hsia has fifteen classes \cite{CostelloHsia}; its class number is therefore a statement in a different integral category.

Likewise, $\Aut_{\cI}(L)$ need not be the full~$\cO_K$-linear orthogonal group of the underlying quadratic lattice.  Since~$L$ is a right~$\cI$-module, right multiplication by~$u\in\cI^1$ preserves its~$\cO_K$-valued quadratic norm but is generally not~$\cI$-linear.  Such enlargements occur naturally in Nebe's golden-lattice viewpoint \cite{NebeGolden}; they do not enter Hashimoto's type-$C_3$ mass.
\end{remark}

\subsection{Hilbert theta normalization}

Let
\[
 \alpha=\frac{5+\sqrt5}{2}=\tau\sqrt5.
\]
The different ideal is~$\dK=(\sqrt5)$, so
$\alpha^{-1}\cO_K=\dK^{-1}$.  For $z=(z_1,\,z_2)\in\mathfrak H^2$,
where~$\mathfrak H$ is the upper half-plane, and~$\nu\in K$, write
\[
 \eK(\nu z)
 =\exp\!\left(2\pi i
 \bigl(\sigma_1(\nu)z_1+\sigma_2(\nu)z_2\bigr)\right).
\]
Let~$\cO_K^+$ consist of the elements of~$\cO_K$ positive at both
real embeddings.  For~$a\in\cO_K^+$, write
\[
 L_a=\{x\in L:h(x,\,x)=a\},
 \qquad r_L(a)=|L_a|.
\]
The scalar theta series of~$L$ is
\begin{equation}\label{eq:scalar-theta}
 \Theta_L(z)
 =\sum_{x\in L}\eK\!\left(\alpha^{-1}h(x,\,x)z\right)
 =1+\sum_{a\in\cO_K^{+}}r_L(a)\,
 \eK(\alpha^{-1}az).
\end{equation}
It is a Hilbert modular form of parallel weight~$6$ for the full level-one Hilbert modular group.  All Fourier coefficients~$r_L(a)$ are nonnegative integers.

Multiplication by the central unit~$\tau$ is a bijection~$L_a\to L_{\tau^2a}$, and therefore
\begin{equation}\label{eq:central-unit-shell-symmetry}
 r_L(a)=r_L(\tau^2a)
 \qquad(a\in\cO_K^+).
\end{equation}
We shall call~$L$ \emph{extremal} if
\begin{equation}\label{eq:extremal-def}
 r_L(1)=0.
\end{equation}
Equivalently, $r_L(\tau^2)=0$.  To see why these are the first two Fourier indices, write $a=m+n\tau\in\cO_K$.  Since
\[
 \Tr_{K/\Q}(\alpha^{-1}a)=m,
\]
the smallest positive diagonal trace index is $m=1$; total positivity then forces $n=0$ or~$1$, giving $a=1$ or $a=1+\tau=\tau^2$.  Thus the single extremality condition removes both norm types of smallest positive trace order in this normalization.  For ordinary modular forms, write~$E_k$ for the normalized Eisenstein
series and~$\Delta$ for the discriminant cusp form.
The term \emph{extremal} is literal: the trace form introduced in \cref{lem:extremal-trace-leech} is even unimodular of rank~$24$ for every~$L\in\cG_3$.  The standard modular-form bound gives a nonzero vector of norm at most~$4$: if the minimum were at least~$6$, then the weight-$12$ theta series would have vanishing coefficients at both~$q$ and~$q^2$, whereas the unique form with constant term~$1$ and vanishing~$q$-coefficient is~$E_4^3-720\Delta$, whose~$q^2$-coefficient is~$196\,560$.  Therefore the smallest positive diagonal trace index is at most~$2$, and condition \eqref{eq:extremal-def} attains the largest possible trace minimum; compare \cite{ConwaySloane,Dursthoff}.

\section{Mass, scalar Hilbert theta series, and extremality}\label{sec:scalar-theta}

\subsection{Hashimoto's mass}

In the Hashimoto--Shimura normalization \cite{Hashimoto,Shimura}, for a totally real field of degree~$n$ and the simply connected inner form of type~$C_m$ attached to a totally definite quaternion algebra split at every finite place, the mass~$\Mass_m$ of the rank-$m$ unimodular genus is
\cite[Proposition~9.3 and equations~(9.1)--(9.2)]{GanHankeYu}
\begin{equation}\label{eq:hashimoto-general}
 \Mass_m
 =2^{-mn}\prod_{j=1}^m|\zeta_K(1-2j)|.
\end{equation}
Here $n=[K:\Q]=2$, so
\begin{equation}\label{eq:hashimoto}
 \Mass_m
 =2^{-2m}\prod_{j=1}^m|\zeta_K(1-2j)|.
\end{equation}
The finite local correction factors are all~$1$: $D$ is split at every finite place and a unimodular lattice has hyperspecial stabilizer there.  The global normalization uses the Tamagawa number~$1$ of the simply connected group~$\mathbf G$.  For $K=\Q(\sqrt5)$, its Dedekind zeta function factors as
$\zeta_K(s)=\zeta(s)L(s,\,\chi_5)$, where~$\chi_5$ is the quadratic
Dirichlet character modulo~$5$.  The generalized Bernoulli-number
formula for~$L(1-n,\,\chi_5)$ gives
\begin{equation}\label{eq:zeta-values}
 \zeta_K(-1)=\frac1{30},
 \qquad
 \zeta_K(-3)=\frac1{60},
 \qquad
 \zeta_K(-5)=\frac{67}{630}.
\end{equation}
As low-rank consistency checks for this normalization,
\[
 \Mass_1=\frac14\,\zeta_K(-1)=\frac1{120}
 =\frac1{|\cI^1|}
\]
and
\[
 \Mass_2=\frac1{16}\,\zeta_K(-1)\zeta_K(-3)
 =\frac1{28\,800}
 =\frac1{|(\cI^1)^2\rtimes S_2|}.
\]
Thus the split lattices already exhaust the masses in ranks~$1$ and~$2$; in particular those two genera have one class.  The rank-$1$ quantity here is the Hermitian-lattice mass~$1/|\cI^1|$; it is not the Eichler mass of the quaternion order, which quotients by central units and equals~$1/60$ in this case.

\begin{remark}[a rank-two shortcut]\label{rem:rank-two-shortcut}
A norm-one vector~$e\in L$ splits off an orthogonal summand $L=e\cI\perp(L\cap e^\perp)$: the projection coefficient~$h(e,\,x)$
lies in~$\cI$, and the unimodular rank-$2$ complement is split by the
preceding mass calculation.  Thus~$L\cong\cI^3$.  We nevertheless retain
the genus-average proof of extremality below, to preserve the
mass-formula and Venkov framework.
\end{remark}

For rank~$3$, consequently
\begin{equation}\label{eq:mass-value}
 \boxed{
 \Mass(\cG_3)
 =\frac1{64}\cdot\frac1{30}\cdot\frac1{60}\cdot\frac{67}{630}
 =\frac{67}{72\,576\,000}.}
\end{equation}

\subsection{The split class}

If $x=(x_1,\,x_2,\,x_3)$ in the split lattice~$\Ls=\cI^3$ has
norm~$1$, then the three totally nonnegative algebraic
integers~$\nrd(x_j)$ sum to~$1$.  Both real embeddings of any
nonzero summand lie in~$(0,1]$, and their product is a positive
integer; hence both equal~$1$.  Thus this summand is~$1$ and the
other two vanish.  The~$360$ norm-one vectors therefore lie exactly
on the three coordinate lines.  These lines are intrinsic, every
isometry is monomial, and
\begin{equation}\label{eq:split-aut}
 \Aut_{\cI}(\Ls)
 \cong(\cI^1)^3\rtimes S_3
 \cong\SL_2(5)^3\rtimes S_3.
\end{equation}
Hence
\begin{equation}\label{eq:split-aut-order}
 |\Aut_{\cI}(\Ls)|
 =120^3\cdot6
 =10\,368\,000,
\end{equation}
and its mass contribution is
\begin{equation}\label{eq:split-mass-contribution}
 \frac1{|\Aut_{\cI}(\Ls)|}
 =\frac1{10\,368\,000}
 =\frac7{72\,576\,000}.
\end{equation}
The two smallest coefficients are
\begin{equation}\label{eq:split-small}
 r_{\Ls}(1)=360,
 \qquad
 r_{\Ls}(\tau^2)=360.
\end{equation}
Indeed, each of the three coordinate lines contributes~$120$ norm-one icosians, and multiplication by the central unit~$\tau$ gives the second shell.

\subsection{The Eisenstein average and coefficient saturation}

For even~$k\ge2$, normalize the Hilbert Eisenstein series by
\begin{equation}\label{eq:hilbert-eisenstein}
 E_{k,\,K}(z)
 =1+\frac4{\zeta_K(1-k)}
 \sum_{a\in\cO_K^+}
 \sigma_{k-1}((a))\,
 \eK(\alpha^{-1}az),
\end{equation}
where
\[
 \sigma_j(\mathfrak n)=\sum_{\mathfrak a\mid\mathfrak n}N(\mathfrak a)^j,
 \qquad N(\mathfrak a)=|\cO_K/\mathfrak a|.
\]
The sum is over integral ideal divisors.
We use Weil's Siegel--Weil formula \cite[Chapter~VI, \S52,
Theorem~5]{WeilSiegel}, in the quaternionic-Hermitian case~$(\mathrm I_1)$ of his \S27.  His parameters are $m=3$, $n=1$,
and $\varepsilon=1/4$, so the convergence condition~$m>2n+4\varepsilon-2$ is satisfied.  Take the characteristic functions
of the self-dual local lattices at the finite places and the Gaussian
for~$\alpha^{-1}h$ at the real places.  Decomposition of the normalized
theta integral over the class set \eqref{eq:adelic-class-set} gives
the weights $\Mass(\cG_3)^{-1}|\Aut_{\cI}(L)|^{-1}$.  The unramified
Eisenstein series has parallel weight~$6$ and constant term~$1$, in the
normalization \eqref{eq:hilbert-eisenstein}.  Thus
\begin{equation}\label{eq:siegel-weil-average}
 \frac1{\Mass(\cG_3)}
 \sum_{[L]\in\cG_3}
 \frac{\Theta_L}{|\Aut_{\cI}(L)|}
 =E_{6,\,K}.
\end{equation}
At a unit ideal the Fourier coefficient of~$E_{6,\,K}$ is
\begin{equation}\label{eq:e6-first-coeff}
 \frac4{\zeta_K(-5)}=\frac{2520}{67}.
\end{equation}

\begin{proposition}[coefficient saturation]\label{prop:coefficient-saturation}
Every class in~$\cG_3$ other than~$[\Ls]$ is extremal.
\end{proposition}

\begin{proof}
The unnormalized genus-average coefficient at $a=1$ is
\[
 \Mass(\cG_3)\frac{2520}{67}
 =\frac{2520}{72\,576\,000}.
\]
The split class alone contributes
\[
 \frac{r_{\Ls}(1)}{|\Aut_{\cI}(\Ls)|}
 =\frac{360}{10\,368\,000}
 =\frac{2520}{72\,576\,000}.
\]
Thus it already accounts for the complete Eisenstein coefficient.  Since every other summand is nonnegative, all other coefficients at $a=1$ vanish.  They are extremal by \eqref{eq:central-unit-shell-symmetry}.
\end{proof}

This proves extremality without first passing to the rank-$24$ trace lattice.

\subsection{The split theta series and a quinary code model}\label{sec:split-quinary-code}

This subsection is not needed for the classification argument; it records the Hirzebruch code--lattice interpretation of the split theta series.  The theta series of the split lattice has the following direct code-theoretic description.  We use the construction explained by Hirzebruch in his Bonn lecture course \emph{Kodierungstheorie und Beziehungen zur Geometrie} in the winter semester 1986/87 \cite{HirzebruchCoding}; cf. also \cite[Chap.~5]{Ebeling}.

Let
\[
 E=\Q(\zeta_5),
 \qquad
 \mathfrak p=(1-\zeta_5)\subset\cO_E,
\]
where~$\zeta_5$ is a primitive fifth root of unity.  The maximal real
subfield is~$K$, and $\cO_E/\mathfrak p\cong\F_5$.  For a code $C\subset\F_5^n$ put
\[
 \Gamma_C
 =\{x\in\cO_E^n:x\bmod\mathfrak p\in C\}.
\]
For a self-orthogonal code, the scaled Hermitian norm
\[
 q_C(x)=\frac1{5}\sum_{j=1}^n x_j\overline{x_j}
\]
has values in~$\mathfrak d_K^{-1}$; here the bar means complex
conjugation on~$E$, fixing~$K$.  The associated Hilbert theta series is obtained from the symmetrized Lee weight enumerator by substituting the three coset theta functions corresponding to
\[
 0,\qquad \{\pm1\},\qquad \{\pm2\}\subset\F_5.
\]

Consider the self-dual length-$2$ code
\begin{equation}\label{eq:quinary-length-two-code}
 C_2=\langle(1,\,2)\rangle_{\F_5}\subset\F_5^2.
\end{equation}
Its symmetrized Lee weight enumerator is
\begin{equation}\label{eq:quinary-length-two-swe}
 \operatorname{swe}_{C_2}(X_0,\,X_1,\,X_2)
 =X_0^2+4X_1X_2.
\end{equation}
The zero code gives the trace lattice~$A_4^2$.  The code~$C_2$ supplies an index-$5$ glue, so the trace lattice of~$\Gamma_{C_2}$ is even unimodular of rank~$8$, hence is~$E_8$.  If~$\vartheta_0$, $\vartheta_1$, and~$\vartheta_2$ denote the three coset theta functions, then
\begin{equation}\label{eq:quinary-e2}
 \Theta_{\Gamma_{C_2}}
 =\vartheta_0^2+4\vartheta_1\vartheta_2.
\end{equation}
Gundlach's structure theorem gives $\dim M_2(\SL_2(\cO_K))=1$.  Since \eqref{eq:quinary-e2} has parallel weight~$2$ and constant term~$1$, it follows that
\begin{equation}\label{eq:rank1-theta}
 \Theta_{\Gamma_{C_2}}=E_{2,\,K}=\Theta_{\cI}.
\end{equation}
Thus~$C_2$ is the cyclotomic code model for one split icosian coordinate at the level of the~$K$-valued norm and its Hilbert theta series; the full right-$\cI$ action is additional structure not seen by the code alone.

Taking three orthogonal copies gives
\[
 C_{\mathrm{spl}}=C_2^{\oplus3}\subset\F_5^6,
 \qquad
 \operatorname{swe}_{C_{\mathrm{spl}}}
 =(X_0^2+4X_1X_2)^3,
\]
and therefore
\begin{equation}\label{eq:split-theta}
 \boxed{\Theta_{\Ls}=E_{2,\,K}^3.}
\end{equation}
Equivalently, the~$A_4^6$ kernel is glued by~$C_2^{\oplus3}$ to the trace lattice~$E_8^3$ of the split icosian lattice.

\subsection{The common nonsplit scalar theta series}

Gundlach determined the graded ring of level-one Hilbert modular forms for~$\Q(\sqrt5)$; Resnikoff and Hirzebruch gave further descriptions, the latter through the Clebsch--Klein cubic surface, and M\"uller and Mayer provide convenient later formulations \cite{Gundlach,Resnikoff,Hirzebruch,Mueller,Mayer}.  In particular, the full graded ring is generated in weights~$2$, $5$, $6$, and~$15$, and
\begin{equation}\label{eq:M6-dimension}
 \dim M_6(\SL_2(\cO_K))=2,
 \qquad
 \dim S_6(\SL_2(\cO_K))=1.
\end{equation}
The two forms~$E_{6,\,K}$ and~$E_{2,\,K}^3$ have constant term~$1$, while their coefficients at $a=1$ are respectively~$2520/67$ and~$360$.  Their difference is therefore a nonzero cusp form, so the conditions ``constant term~$1$'' and ``coefficient at $a=1$ equal to~$0$'' are independent.  Hence there is a unique such weight-$6$ form; denote it by~$\Theta_{\mathrm{ext}}$.  \Cref{prop:coefficient-saturation} shows that every nonsplit class has this scalar theta series.  Since the normalized mass weight of the split class is
\[
 \frac{|\Aut_{\cI}(\Ls)|^{-1}}{\Mass(\cG_3)}=\frac7{67},
\]
the total normalized mass of the nonsplit part is~$60/67$, and the Siegel--Weil average gives
\begin{equation}\label{eq:extremal-theta}
 \boxed{
 \Theta_{\mathrm{ext}}
 =\frac{67E_{6,\,K}-7E_{2,\,K}^3}{60}.}
\end{equation}
Here~$60/67$ is the aggregate mass weight of all nonsplit classes.  Thus the automorphic average determines the scalar theta series of any concrete nonsplit lattice before either the class number or its full automorphism group is known.

\subsection{The first extremal shell coefficients}\label{sec:extremal-shell-counts}
From $\Theta_{\cI}=E_{2,\,K}$ and the ideal norms $N((2))=4$ and $N((2+\tau))=5$, one obtains
\[
 r_{\cI}(1)=r_{\cI}(\tau^2)=120,
 \qquad
 r_{\cI}(2)=600,
 \qquad
 r_{\cI}(2+\tau)=720.
\]
Therefore
\[
 r_{\Ls}(2)=3\cdot600+3\cdot120^2=45\,000,
 \qquad
 r_{\Ls}(2+\tau)=3\cdot720+6\cdot120^2=88\,560,
\]
where the two-coordinate decompositions in the second formula are $1+\tau^2=2+\tau$.  For a Hilbert modular form~$f$, let~$[a]f$ denote its coefficient of
$\eK(\alpha^{-1}az)$.  The corresponding Eisenstein coefficients are
\[
 [2]E_{6,\,K}=\frac{2520}{67}(1+4^5),
 \qquad
 [2+\tau]E_{6,\,K}=\frac{2520}{67}(1+5^5).
\]
Taking coefficients in \eqref{eq:extremal-theta} gives
\begin{equation}\label{eq:extremal-small-counts}
 \boxed{
 r_{\mathrm{ext}}(2)=37\,800,
 \qquad
 r_{\mathrm{ext}}(2+\tau)=120\,960,
 \qquad
 r_{\mathrm{ext}}(2\tau^2)=37\,800.}
\end{equation}
The following arithmetic observation identifies the projective shell
cardinalities without using the design identities.

\begin{lemma}[primitive shell lines]\label{lem:primitive-shell-lines}
Let~$L\in\cG_3$ be extremal, $a\in\{2,\,2+\tau,\,2\tau^2\}$, and~$x\in L_a$.  Then $L\cap xD=x\cI$, and the norm-$a$ vectors on
this line are exactly~$x\cI^1$.  Thus the three projective shells~$L_a/\cI^1$ are sets of~$315$, $1008$, and~$315$ points.
\end{lemma}

\begin{proof}
Every fractional right ideal of the icosian order is principal
\cite[\S2]{BaakeGrimmHeuerZeiner}.  Hence $\mathfrak a=\{q\in D:xq\in L\}=c\cI$ for some~$c\in D^\times$.
Since~$1\in\mathfrak a$, we have~$c^{-1}\in\cI$, and
\[
 b=\nrd(c^{-1})\in\cO_K^+,
 \qquad h(xc,\,xc)=a/b\in\cO_K^+.
\]
The elements~$2$ and~$2+\tau$ are prime in~$\cO_K$ (the former
is inert and the latter has field norm~$5$), and~$2\tau^2$ is
associated to~$2$.  Thus either~$b$ is a unit, in which case~$c^{-1}\in\cI^\times$ and $\mathfrak a=\cI$, or~$a/b$ is a
totally positive unit.  Every such unit is~$\tau^{2k}$; then~$xc\tau^{-k}$ has norm~$1$, contrary to extremality.
The norm-$a$ vectors in~$x\cI$ are precisely~$x\cI^1$, an orbit
of size~$120$.  Two vectors in~$D^3$ on the same real quaternionic
line have their ratio in~$D$, as seen from any nonzero coordinate.
Therefore \eqref{eq:extremal-small-counts} gives the asserted
cardinalities at either real place.  Multiplication by~$\tau$
identifies the first and third sets projectively.
\end{proof}

\section{The icosian Leech lattice and quaternionic projective designs}\label{sec:rank3-venkov}

The scalar Hilbert theta series supplied the shell counts used in the classification.  We now pass to their one-place spherical refinements.  The problem is to determine three trace-minimal Fourier coefficients; their vanishing is equivalent to the projective design statement.

\subsection{Quaternionic projective harmonics and designs}\label{sec:projective-harmonics}

Fix the real embedding~$\sigma_1$; the same arguments apply at~$\sigma_2$.
At either place, $D\otimes_K\R\cong\Hh$.  For~$v\in D^3$, write
$v^{(i)}\in\Hh^3$ for its image at~$\sigma_i$.
For a nonzero vector~$x\in\Hh^3$, let
\[
 p_x=\frac{xx^*}{h(x,\,x)}\in\Herm_3(\Hh)
\]
be the rank-one orthogonal projector onto the right line~$x\Hh$.
Its projective point is~$[x]\in\Hh\PP^2$; we also identify this point
with~$p_x$.  For a quaternion~$q\in\Hh$, its squared absolute value is
$|q|^2=q\overline q$.  For a quaternionic matrix, $\tr$ denotes the sum of its diagonal entries; this need not be real
unless the matrix is Hermitian.  On the real Jordan algebra
$J=\Herm_3(\Hh)$ put $\Tr_J(a)=\tr(a)\in\R$ and
\[
 T(a,\,b)=\operatorname{Re}\tr(ab)
       =\Tr_J\!\left(\frac{ab+ba}{2}\right).
\]
Thus~$\Tr_J$, the field trace~$\Tr_{K/\Q}$, and the reduced
quaternionic trace~$\trd$ have distinct meanings.  The projective
inner product is
\begin{equation}\label{eq:projective-inner-product}
 t([x],\,[y])=\operatorname{Re}\tr(p_xp_y)=T(p_x,\,p_y)
 =\frac{|h(x,\,y)|^2}{h(x,\,x)h(y,\,y)}
 \in[0,\,1].
\end{equation}

For~$r\ge0$, let~$\cH_r$ be the space of right-$\Sp(1)$-invariant real harmonic
homogeneous polynomials of degree~$2r$ on~$\Hh^3$.  It is an irreducible
spherical~$\Sp(3)$-module.  Its elements give projective harmonics by
\[
 \widehat P([x])=\frac{P(x)}{h(x,\,x)^r}
 \qquad(P\in\cH_r,\,x\ne0).
\]
Write~$\mu$ for the invariant probability measure on~$\Hh\PP^2$.
For these harmonic spaces and their zonal polynomials, see
\cite{MohammadpourWaldron}.

\begin{definition}\label{def:projective-design}
A finite set $X\subset\Hh\PP^2$ is a quaternionic projective~$s$-design if
\[
 \frac1{|X|}\sum_{x\in X}P(x)
 =\int_{\Hh\PP^2}P\,d\mu
\]
for every polynomial of degree at most~$s$ in the real entries of the
projector~$p_x$.  Equivalently,
\[
 \sum_{x\in X}\widehat P(x)=0
 \qquad(P\in\cH_r,\,1\le r\le s).
\]
\end{definition}
The same definition applies to finite multisets, with every sum taken with multiplicity.  This is Hoggar's Delsarte-space definition; its equivalence with the quaternionic spherical~$(s,\,s)$-design formulation and its relation to quaternionic tight frames are discussed in \cite{HoggarDesigns,Waldron,WaldronFrames}.

For fixed~$x$ and Haar-random~$y$, the variable~$t(x,\,y)$ has beta
distribution with parameters~$2$ and~$4$.  With the rising factorials
$(a)_0=1$ and $(a)_k=a(a+1)\cdots(a+k-1)$ for~$k>0$, its moments are
\begin{equation}\label{eq:haar-moments}
 \int_{\Hh\PP^2}t(x,\,y)^k\,d\mu(y)
 =\frac{(2)_k}{(6)_k}
 \qquad(k\ge0).
\end{equation}
Using the Jacobi polynomials~$P_r^{(3,\,1)}$, define the normalized zonal polynomials
\begin{equation}\label{eq:projective-zonal-polynomials}
 Q_r(t)=\frac{P_r^{(3,\,1)}(2t-1)}{\binom{r+3}{3}},
 \qquad Q_r(1)=1.
\end{equation}

\subsection{Spherical Hilbert theta coefficients}\label{sec:spherical-theta}

Let~$P\in\cH_r$, viewed as a right-$\Sp(1)$-invariant harmonic polynomial of ordinary degree~$2r$ on the first real component of~$D^3$.  Define
\begin{equation}\label{eq:weighted-hilbert-theta}
 \Theta^{(1)}_{L,\,P}(z_1,\,z_2)
 =\sum_{v\in L}P(v^{(1)})
 \eK\!\left(\alpha^{-1}h(v,\,v)z\right).
\end{equation}
Standard harmonic-theta transformation gives
\begin{equation}\label{eq:weighted-hilbert-weight}
 \Theta^{(1)}_{L,\,P}
 \in S_{(6+2r,\,6)}(\SL_2(\cO_K))
 \qquad(r>0)
\end{equation}
\cite{Coulangeon94,Bachoc}.  This nonparallel Hilbert weight is the natural automorphic home for a one-place spherical coefficient of the compact group of type~$C_3$.

For general~$a\in\cO_K^+$, the quotient~$L_a/\cI^1$ is a
projective multiset, with one occurrence per right-unit orbit.
For extremal~$L$ and the three trace-minimal norms it is a set by
\cref{lem:primitive-shell-lines}.  For~$x\in L_a$, the homogeneous
normalization gives
\[
 \widehat P([x^{(1)}])
 =\frac{P(x^{(1)})}{\sigma_1(a)^r}.
\]

\begin{lemma}[Fourier coefficients and projective shells]\label{lem:coefficient-design}
The Fourier coefficient of \eqref{eq:weighted-hilbert-theta} at~$a$ is
\begin{equation}\label{eq:coefficient-projective-shell}
 \sum_{v\in L_a}P(v^{(1)})
 =120\,\sigma_1(a)^r
  \sum_{\ell\in L_a/\cI^1}\widehat P(\ell).
\end{equation}
Consequently, for a fixed~$s$, the coefficients at~$a$ vanish for every~$P\in\cH_r$ and~$1\le r\le s$ if and only if the projective shell multiset~$L_a/\cI^1$ is a quaternionic projective~$s$-design.
\end{lemma}

\begin{proof}
Right multiplication by~$\cI^1$ acts freely on every nonzero shell, and each orbit has size $|\cI^1|=120$.  The polynomial~$P$ is right-$\Sp(1)$-invariant and homogeneous of degree~$2r$, so it is constant on such an orbit and differs from~$\widehat P$ by the factor~$\sigma_1(a)^r$.  Formula \eqref{eq:coefficient-projective-shell} follows, and the last assertion is exactly the criterion in \cref{def:projective-design}.
\end{proof}

We shall prove vanishing at $a=2$, $2+\tau$, and~$2\tau^2$ for~$1\le r\le5$.  This cannot follow merely from vanishing of the
ambient mixed-weight cusp-form space.  Indeed, for the split lattice
the coefficient at $a=1$ is $120\sum_{j=1}^3\widehat P([e_j])$.
For a zonal harmonic about~$[e_1]$, this sum is
\[
 1+2Q_r(0)=1+\frac{12(-1)^r}{(r+2)(r+3)}\ne0
 \qquad(r\ge2).
\]
Thus $S_{(6+2r,\,6)}(\SL_2(\cO_K))\ne0$ for every~$r\ge2$;
the trace-lattice argument below supplies the additional constraints.

\subsection{The trace lattice and the three trace-minimal Hermitian shells}

\begin{lemma}[the extremal trace lattice]\label{lem:extremal-trace-leech}
Let~$L$ be a unimodular right~$\cI$-lattice of rank~$3$.  On the underlying~$\Z$-module define
\begin{equation}\label{eq:trace-bilinear-form}
 B_L(x,\,y)=\Tr_{K/\Q}\!\left(\alpha^{-1}\trd(h(x,\,y))\right).
\end{equation}
Then~$(L,\,B_L)$ is an even unimodular lattice of rank~$24$.  If~$L$ is extremal, this lattice has no roots and hence is the Leech
lattice~$\Lambda$.  Writing~$\Min$ for the set of nonzero vectors of
minimum norm, its minimal shell is the disjoint union
\begin{equation}\label{eq:leech-shell-splitting}
 \Min(L,\,B_L)
 =L_2\sqcup L_{2+\tau}\sqcup L_{2\tau^2}.
\end{equation}
\end{lemma}

\begin{proof}
The quaternion algebra~$D$ has trivial finite discriminant, so a maximal order is self-dual for the reduced-trace pairing.  Together with $\alpha^{-1}\cO_K=\mathfrak d_K^{-1}$ and the Hermitian unimodularity of~$L$, this shows that \eqref{eq:trace-bilinear-form} is integral and unimodular.  Since~$h(x,\,x)\in K$,
\[
 B_L(x,\,x)=2\Tr_{K/\Q}\!\left(\alpha^{-1}h(x,\,x)\right)\in2\Z,
\]
so the lattice is even.  Its rank is $3\cdot4\cdot2=24$.

Assume now that~$L$ is extremal.  Write $a=m+n\tau\in\cO_K$.  The identity
\begin{equation}\label{eq:trace-index-coordinate}
 \Tr_{K/\Q}(\alpha^{-1}a)=m
\end{equation}
follows from $\Tr_{K/\Q}(\alpha^{-1})=1$ and $\Tr_{K/\Q}(\alpha^{-1}\tau)=0$.  If~$a$ is totally positive and $m=1$, then $n=0$ or~$1$, so $a=1$ or $a=\tau^2$.  Extremality therefore excludes vectors of~$B_L$-norm~$2$.

If~$a$ is totally positive and $m=2$, then~$-1\le n\le3$, and the five possibilities are
\[
 2-\tau=\tau^{-2},\qquad
 2,\qquad
 2+\tau=\alpha,\qquad
 2+2\tau=2\tau^2,\qquad
 2+3\tau=\tau^4.
\]
Multiplication by the central unit~$\tau$ sends~$L_a$ bijectively to~$L_{\tau^2a}$.  Hence $L_{\tau^{-2}}=\varnothing$ because $L_1=\varnothing$, and $L_{\tau^4}=\varnothing$ because $L_{\tau^2}=\varnothing$.  This proves \eqref{eq:leech-shell-splitting}.  The unique rootless even unimodular lattice of rank~$24$ is the Leech lattice \cite[Ch.~24]{ConwaySloane}.
\end{proof}

\begin{remark}[Forgetting to the trace lattice]\label{rem:trace-forgetful}
The passage~$(L,\,h)\mapsto(L,\,B_L)$ forgets the~$\cO_K$- and~$\cI$-module structures.  It gives~$E_8^3$ for the split class, as seen in \cref{sec:split-quinary-code}, and the Leech lattice for every extremal class; Nebe's golden-lattice viewpoint retains the intermediate~$\cO_K$-structure \cite{NebeGolden}.  The trace lattice sees only the union in \eqref{eq:leech-shell-splitting}, so its harmonic identities must be supplemented by the two-place radial separation below to distinguish the three Hermitian norm types.
\end{remark}

\begin{remark}[diagonal restriction]\label{rem:diagonal-theta-check}
Setting $z_1=z_2=z$ in \eqref{eq:scalar-theta} gives the ordinary theta series of the trace lattice:
\[
 \Theta_L(z,\,z)=\sum_{x\in L}q^{B_L(x,\,x)/2},
 \qquad q=e^{2\pi iz}.
\]
Consequently
\begin{align*}
 \Theta_{\Ls}(z,\,z)&=\Theta_{E_8}(z)^3=E_4(z)^3
   =1+720q+179\,280q^2+\cdots,\\
 \Theta_{\mathrm{ext}}(z,\,z)&=\Theta_{\Lambda}(z)
   =E_4(z)^3-720\Delta(z)
   =1+196\,560q^2+16\,773\,120q^3+\cdots.
\end{align*}
This supplies a direct consistency check on the Hilbert-theta normalization and on the first shell coefficients.
\end{remark}

The shell sizes from \eqref{eq:extremal-small-counts} give the familiar numerical decomposition
\begin{equation}\label{eq:leech-minimal-decomposition-count}
 196\,560=37\,800+120\,960+37\,800.
\end{equation}
The point is to separate these three pieces at the level of harmonic sums.

\begin{lemma}[the required Leech harmonic vanishings]\label{lem:leech-harmonic-vanishing}
Let~$(\Lambda,\,B)$ be the Leech lattice.  If~$H$ is a harmonic homogeneous polynomial on~$\R^{24}$ of degree
\[
 d\in\{2,\,4,\,6,\,8,\,10,\,14\},
\]
then
\begin{equation}\label{eq:leech-harmonic-shell-zero}
 \sum_{x\in\Min(\Lambda)}H(x)=0.
\end{equation}
\end{lemma}

\begin{proof}
The weighted theta series
\[
 \Theta_{\Lambda,\,H}(z)
 =\sum_{x\in\Lambda}H(x)q^{B(x,\,x)/2}
\]
is a cusp form of weight~$12+d$ for~$\SL_2(\Z)$.  The Leech lattice has no vectors of norm~$2$, so its coefficient of~$q$ is zero.  Now
\[
 S_{14}(\SL_2(\Z))=0,
 \qquad
 S_{12+d}(\SL_2(\Z))=\C\,\Delta E_d
 \quad(d=4,\,6,\,8,\,10,\,14),
\]
and every displayed generator~$\Delta E_d$ has nonzero coefficient of~$q$.  Thus $\Theta_{\Lambda,\,H}=0$ in all six degrees, and in particular its coefficient on the minimal shell vanishes.  This is the usual Venkov argument for the spherical~$11$-design property of the Leech shell, together with its additional degree-$14$ vanishing \cite{Venkov,Bachoc,Pache}.
\end{proof}

\subsection{Harmonic separation at the two real places}

Let~$W_1$ and~$W_2$ be the two real components of~$L\otimes_\Z\R$, each identified with~$\Hh^3$.  We rescale the trace form globally by~$1/2$ and write a lattice vector as $(X_v,\,Y_v)\in W_1\oplus W_2$ so that
\begin{equation}\label{eq:two-place-radii}
 \varrho_1(v)=\lVert X_v\rVert^2
 =\sigma_1(\alpha^{-1}h(v,\,v)),
 \qquad
 \varrho_2(v)=\lVert Y_v\rVert^2
 =\sigma_2(\alpha^{-1}h(v,\,v)).
\end{equation}
This global rescaling does not change harmonicity or spherical-design identities.  Put
\[
 r_-=1-\frac1{\sqrt5},
 \qquad
 r_+=1+\frac1{\sqrt5}.
\]
For the three pieces of \eqref{eq:leech-shell-splitting}, the component radii are
\begin{equation}\label{eq:three-radius-table}
\begin{array}{c|cc}
 a&\varrho_1&\varrho_2\\ \hline
 2&r_-&r_+\\
 2+\tau&1&1\\
 2\tau^2&r_+&r_-.
\end{array}
\end{equation}

\begin{remark}[a radial check of the shell counts]\label{rem:radial-count-check}
The quartic $\varrho_1^2-\frac73\varrho_1\varrho_2+\varrho_2^2$
is harmonic: its Laplacian is $56(\varrho_1+\varrho_2)-\frac73\,24(\varrho_1+\varrho_2)=0$.
Its values on the three rows of \eqref{eq:three-radius-table} are~$8/15$, $-1/3$, and~$8/15$.  Thus, with $A=r_L(2)=r_L(2\tau^2)$ and $B=r_L(2+\tau)$, the ordinary harmonic-theta identity gives
\[
 \frac{16}{15}A-\frac13B=0,
 \qquad 2A+B=196\,560,
 \qquad (A,\,B)=(37\,800,\,120\,960).
\]
This checks the Hilbert-modular calculation using only the ordinary
theta series and the degree-$4$ vanishing of
\cref{lem:leech-harmonic-vanishing}; both follow from even
unimodularity and absence of roots, without Leech uniqueness.
\end{remark}

Let~$P\in\cH_r$, realized as a right-$\Sp(1)$-invariant harmonic polynomial of ordinary degree~$2r$ on~$W_1$.  A fixed scalar rescaling of its argument is immaterial, so we use the same symbol~$P$ in the coordinates of \eqref{eq:two-place-radii}.  Define
\begin{equation}\label{eq:ABC-harmonic-sums}
 A_r=\sum_{v\in L_2}P(X_v),
 \qquad
 B_r=\sum_{v\in L_{2+\tau}}P(X_v),
 \qquad
 C_r=\sum_{v\in L_{2\tau^2}}P(X_v).
\end{equation}
Multiplication by the central unit~$\tau$ gives a bijection~$L_2\to L_{2\tau^2}$ and multiplies the first real component by~$\tau$.  Homogeneity therefore gives
\begin{equation}\label{eq:C-tau-A}
 C_r=\tau^{2r}A_r.
\end{equation}
Since~$P(X)$ is harmonic on~$W_1\oplus W_2$ and has degree~$2r\le10$, \cref{lem:leech-harmonic-vanishing} gives
\begin{equation}\label{eq:first-separation-equation}
 A_r+B_r+C_r=0.
\end{equation}

For~$1\le r\le4$, put $\beta_r=3/(r+3)$ and consider
\begin{equation}\label{eq:low-degree-separator}
 H_r(X,\,Y)=P(X)\bigl(\varrho_2-\beta_r\varrho_1\bigr).
\end{equation}
Here $\dim_\R W_1=\dim_\R W_2=12$, and Euler's identity gives
\[
 \Delta_X(\varrho_1P)=8(r+3)P,
 \qquad
 \Delta_Y(\varrho_2P)=24P.
\]
Thus~$H_r$ is harmonic.  Its degree is~$2r+2\in\{4,\,6,\,8,\,10\}$, so summing it over the Leech minimal shell yields
\begin{equation}\label{eq:second-separation-equation}
 (r_+-\beta_r r_-)A_r+(1-\beta_r)B_r
 +(r_--\beta_r r_+)C_r=0.
\end{equation}
Substituting \eqref{eq:C-tau-A} and \eqref{eq:first-separation-equation} into \eqref{eq:second-separation-equation} gives
\[
 \frac{1+\beta_r}{\sqrt5}(1-\tau^{2r})A_r=0.
\]
The coefficient is nonzero, and hence
\begin{equation}\label{eq:ABC-zero-low}
 A_r=B_r=C_r=0
 \qquad(1\le r\le4).
\end{equation}

For $r=5$, a quadratic radial factor would have total degree~$12$, where the Leech minimal shell need not annihilate all harmonics.  Instead use the degree-$14$ harmonic
\begin{equation}\label{eq:degree14-separator}
 H_5(X,\,Y)=P(X)\mathcal R(\varrho_1,\,\varrho_2),
 \qquad
 \mathcal R(u,\,v)=v^2-\frac78uv+\frac{21}{136}u^2.
\end{equation}
Indeed,
\[
 \Delta(P\varrho_2^2)=56P\varrho_2,
 \qquad
 \Delta(P\varrho_1\varrho_2)=64P\varrho_2+24P\varrho_1,
 \qquad
 \Delta(P\varrho_1^2)=136P\varrho_1,
\]
so $\Delta H_5=0$.  The degree-$14$ case of \cref{lem:leech-harmonic-vanishing} gives
\begin{equation}\label{eq:degree14-separation-equation}
 \mathcal R(r_-,\,r_+)A_5+\mathcal R(1,\,1)B_5+\mathcal R(r_+,\,r_-)C_5=0.
\end{equation}
The three values are
\begin{align*}
 \mathcal R(1,\,1)&=\frac{19}{68},\\
 \mathcal R(r_-,\,r_+)&=\frac{233}{340}+\frac{23\sqrt5}{68},\\
 \mathcal R(r_+,\,r_-)&=\frac{233}{340}-\frac{23\sqrt5}{68}.
\end{align*}
Using \eqref{eq:C-tau-A} and \eqref{eq:first-separation-equation}, equation \eqref{eq:degree14-separation-equation} becomes $D_5A_5=0$, where
\begin{align}
 D_5
 &=\mathcal R(r_-,\,r_+)+\tau^{10}\mathcal R(r_+,\,r_-)
   -(1+\tau^{10})\mathcal R(1,\,1)\notag\\
 &=-\frac{115}{136}(25+11\sqrt5)\ne0.
 \label{eq:degree14-determinant}
\end{align}
Consequently
\begin{equation}\label{eq:ABC-zero-five}
 A_5=B_5=C_5=0.
\end{equation}

\begin{theorem}[rank-$3$ icosian Venkov theorem]\label{thm:rank3-venkov}
Let~$L$ be an extremal unimodular icosian lattice of quaternionic rank~$3$.  For every~$P\in\cH_r$ with~$1\le r\le5$, the three trace-minimal Fourier coefficients of \eqref{eq:weighted-hilbert-theta} vanish:
\begin{equation}\label{eq:three-weighted-zero}
 \sum_{v\in L_2}P(v^{(1)})
 =\sum_{v\in L_{2+\tau}}P(v^{(1)})
 =\sum_{v\in L_{2\tau^2}}P(v^{(1)})
 =0.
\end{equation}
Equivalently, for every $a\in\{2,\,2+\tau,\,2\tau^2\}$, the projective shell $L_a/\cI^1\subset\Hh\PP^2$ is a quaternionic projective~$5$-design.
\end{theorem}

\begin{proof}
Up to the fixed nonzero rescaling used in \eqref{eq:two-place-radii}, the quantities~$A_r$, $B_r$, and~$C_r$ in \eqref{eq:ABC-harmonic-sums} are exactly the three Fourier coefficients in \eqref{eq:three-weighted-zero}.  Their vanishing follows from \eqref{eq:ABC-zero-low} and \eqref{eq:ABC-zero-five}.  The equivalent projective statement follows from \cref{lem:coefficient-design}.
\end{proof}

\begin{remark}[Scope and relation to Venkov's method]\label{rem:venkov-scope}
The theorem determines only the three trace-minimal coefficients of
\eqref{eq:weighted-hilbert-theta}, not the entire weighted Hilbert
theta series.  Once the angle distribution is known,
\cref{cor:trace-shell-harmonic-strength} adds degree-$7$ vanishing
and proves that the design strength is exactly~$5$.  The proof is specific to rank~$3$, where exactly three Hermitian norm types survive at trace index~$2$ and can be separated by the two archimedean radii.  The harmonic-theta and trace-lattice ingredients are classical \cite{Venkov,Bachoc,Pache,NebeGolden,Dursthoff}; the additional step used here is the explicit two-place radial separation of the three icosian components of the Leech minimal shell.
\end{remark}

\section{The \texorpdfstring{norm-$2$}{norm-2} shell: angles, moments, and frames}\label{sec:norm2-geometry}

Let~$L\in\cG_3$ be extremal.  By \cref{lem:primitive-shell-lines,thm:rank3-venkov},
\[
 X_L=L_2/\cI^1
\]
is a~$315$-point quaternionic projective~$5$-design.

\subsection{The five projective inner products and valencies}

\begin{proposition}[the possible projective inner products]\label{prop:extremal-shell-angles}
If~$[x]$ and~$[y]$ are distinct projective points occurring in~$X_L$, then
\begin{equation}\label{eq:five-angle-set}
 t([x],\,[y])\in
 \left\{0,\,\frac12,\,\frac14,\,g_+,\,g_-\right\},
 \qquad
 g_\pm=\frac{3\pm\sqrt5}{8}.
\end{equation}
\end{proposition}

\begin{proof}
Choose representatives~$x$ and~$y$ in~$L_2$ and put
\[
 q=h(x,\,y)\in\cI,
 \qquad
 a=\nrd(q)\in\cO_K.
\]
Here~$q\in\cI$ because $L=L^\#$.  At each real place, Cauchy--Schwarz gives
\[
 0\le a\le4,
 \qquad
 0\le a'\le4.
\]
Writing $a=m+n\tau$ and using $a-a'=n\sqrt5$ gives~$|n|\le1$, and hence
\begin{equation}\label{eq:bounded-icosian-norms}
 a\in
 \{0,\,1,\,2,\,3,\,4,\,\tau^2,\,\tau^{-2},\,2+\tau,\,3-\tau\}.
\end{equation}
The value $a=4$ can occur only when $[x]=[y]$.  Indeed, put
\[
 \lambda=h(x,\,x)^{-1}h(x,\,y)
 \qquad\text{and}\qquad
 z=y-x\lambda.
\]
Then $h(z,\,z)=2-a/2=0$, so definiteness gives $z=0$ and $y=x\lambda$.

It remains to exclude
\[
 a\in\{3,\,2+\tau,\,3-\tau\}.
\]
Put $\varepsilon=4-a$.  In these three cases choose respectively
\[
 c=1,\qquad c=\tau^{-1},\qquad c=\tau;
\]
then $c^2=\varepsilon\in\{1,\,\tau^{-2},\,\tau^2\}$.  Under~$\cI/2\cI\cong M_2(\F_4)$, reduced norm becomes determinant.  Since $a\equiv4-a=\varepsilon=c^2\pmod{2\cO_K}$, the reduction of~$c^{-1}q$ has determinant~$1$.  The surjectivity in \eqref{eq:reduction-mod2} therefore gives~$u\in\cI^1$ with
\[
 u\equiv c^{-1}q\pmod{2\cI}.
\]
Write $q-cu=2r$ with~$r\in\cI$ and set $z=y-xr\in L$.  A direct calculation gives
\begin{align*}
 2h(z,\,z)
 &=4-2\trd(\overline rq)+4\nrd(r)\\
 &=4-a+\nrd(q-2r)
 =\varepsilon+\nrd(cu)
 =2\varepsilon.
\end{align*}
Thus $h(z,\,z)=\varepsilon$, and~$z\ne0$ because~$\varepsilon$ is totally positive.  This contradicts extremality when $\varepsilon=1$ or~$\tau^2$; when $\varepsilon=\tau^{-2}$, multiplication by~$\tau$ gives a vector of norm~$1$.  The three unwanted values are impossible.  Dividing the remaining reduced norms by $h(x,\,x)h(y,\,y)=4$ proves \eqref{eq:five-angle-set}.
\end{proof}

\begin{lemma}[valencies]\label{lem:norm2-multiplicity-free}
For every~$x\in X_L$, the five projective inner products
to the other points have multiplicities
\begin{equation}\label{eq:angle-multiplicity-table}
\begin{array}{c|ccccc}
 t(x,\,y)&0&\frac12&\frac14&g_+&g_-\\ \hline
 \#\{y:t(x,\,y)=\cdot\}&10&80&160&32&32.
\end{array}
\end{equation}
\end{lemma}

\begin{proof}
Let~$n_0$, $n_{1/2}$, $n_{1/4}$, $n_+$, and~$n_-$ be the five multiplicities.
There are no repeated projective points by
\cref{lem:primitive-shell-lines}, so their sum is~$314$.
By \cref{thm:rank3-venkov}, for~$1\le k\le4$,
\begin{equation}\label{eq:one-point-valency-moments}
 1+n_{1/2}2^{-k}+n_{1/4}4^{-k}+n_+g_+^k+n_-g_-^k
 =315\frac{(2)_k}{(6)_k}.
\end{equation}
Together with the cardinality equation, these form the Vandermonde
system on the five distinct angle values.  Substitution verifies the
displayed multiplicities, which are therefore its unique solution.
\end{proof}

\begin{lemma}[Hoggar's bound and its equality case]\label{lem:hoggar-bound-equality}
A subset $Y\subset\Hh\PP^2$ whose projective inner products between
distinct points belong to \eqref{eq:five-angle-set} has at most~$315$
points.  If $|Y|=315$, then~$Y$ is a quaternionic projective~$5$-design.
\end{lemma}

\begin{proof}
The zonal polynomials \eqref{eq:projective-zonal-polynomials} have
the beta weight in \eqref{eq:haar-moments}.
Hoggar's annihilator \eqref{eq:hoggar-polynomial-intro} has the expansion
\begin{align}
 F(t)={}&\frac1{315}Q_0(t)+\frac{13}{495}Q_1(t)
        +\frac{113}{1155}Q_2(t)\notag\\
       &+\frac{16}{65}Q_3(t)+\frac{64}{165}Q_4(t)
        +\frac{512}{2145}Q_5(t).
 \label{eq:hoggar-positive-expansion}
\end{align}
All six coefficients~$f_r$ are strictly positive.  For a finite set~$Y$ of size~$N$, the addition formula for projective harmonics gives
\[
 S_r(Y):=\sum_{x,\,y\in Y}Q_r(t(x,\,y))\ge0,
\]
and $S_r(Y)=0$ if and only if $\sum_{y\in Y}\widehat P(y)=0$ for every
$P\in\cH_r$; see \cite{HoggarDesigns}.  Since $F(1)=1$ and~$F$ vanishes at the
five allowed off-diagonal values,
\[
 N=\sum_{x,\,y\in Y}F(t(x,\,y))
  =\frac{N^2}{315}+\sum_{r=1}^5 f_r S_r(Y)
  \ge\frac{N^2}{315}.
\]
This proves~$N\le315$.  At equality all~$S_r(Y)$, $1\le r\le5$,
vanish, which is precisely the projective~$5$-design condition.
An infinite set is also excluded by applying the finite bound to any~$316$ of its points.
\end{proof}

\begin{corollary}[harmonic strength of the three shells]\label{cor:trace-shell-harmonic-strength}
At either real place, the three trace-minimal projective shells of an
extremal~$L\in\cG_3$ are exactly~$5$-designs and also annihilate~$\cH_7$.
\end{corollary}

\begin{proof}
For~$X_L$, the valencies give
\[
 \frac{S_r(X_L)}{315}
 =1+10Q_r(0)+80Q_r(1/2)+160Q_r(1/4)
   +32Q_r(g_+)+32Q_r(g_-).
\]
Substitution into \eqref{eq:projective-zonal-polynomials} gives $S_6(X_L)/315=195/64>0$ and $S_7(X_L)=0$.
By the addition formula used in \cref{lem:hoggar-bound-equality}, $X_L$ annihilates~$\cH_7$ but not~$\cH_6$.

Use the sums~$A_r$, $B_r$, and~$C_r$ of \eqref{eq:ABC-harmonic-sums}.
For $r=7$, the polynomial~$P(X)$ has ordinary harmonic degree~$14$;
hence \cref{lem:leech-harmonic-vanishing} gives $A_7+B_7+C_7=0$.  The preceding vanishing and $C_7=\tau^{14}A_7$ therefore give $B_7=0$.
For $r=6$, the polynomial $P(X)(\varrho_2-\varrho_1/3)$ is harmonic of degree~$14$, since $\Delta(P\varrho_1)=72P$ and $\Delta(P\varrho_2)=24P$.
Its shell sum gives
\[
 B_6=-\frac32\bigl[(r_+-r_-/3)
       +\tau^{12}(r_--r_+/3)\bigr]A_6
     =-(18+8\sqrt5)A_6.
\]
Some~$A_6$ is nonzero, so the middle shell also fails to be a~$6$-design.  The third shell is projectively the first.
At the other real place the valencies are unchanged and the
nonzero scalar is conjugated.  Together with
\cref{thm:rank3-venkov}, this proves the assertion.
\end{proof}

For~$X_L$, arithmetic primitivity, the shell count, and the five-angle
restriction also allow Hoggar equality to recover the~$5$-design
property.  The Venkov argument above explains it simultaneously for
all three shells.

For the remainder of the combinatorial argument, let $Y\subset\Hh\PP^2$
be any set of~$315$ points whose projective inner products between
distinct points belong to \eqref{eq:five-angle-set}.  It is a
quaternionic projective~$5$-design by
\cref{lem:hoggar-bound-equality}.  The set~$X_L$ satisfies
these hypotheses for every extremal~$L$.

\subsection{Mixed moments and the five-class association scheme}

Put
\begin{equation}\label{eq:angle-nodes}
 \rho_1=0,
 \qquad \rho_2=\frac12,
 \qquad \rho_3=\frac14,
 \qquad \rho_4=g_+,
 \qquad \rho_5=g_-.
\end{equation}
For distinct~$x$, $y\in Y$ with $t(x,\,y)=\rho_k$, define
\begin{equation}\label{eq:angle-intersection-numbers}
 p_{ij}(x,\,y)
 =\#\{z\in Y:t(x,\,z)=\rho_i,\,t(y,\,z)=\rho_j\},
 \qquad 1\le i,\,j\le5.
\end{equation}

\begin{lemma}[two-point Haar moments]\label{lem:quaternionic-two-point-moments}
Let~$x$ and~$y$ be points of~$\Hh\PP^2$ and put $s=t(x,\,y)$.  For nonnegative integers~$a$ and~$b$, one has
\begin{equation}\label{eq:two-point-haar-coefficient}
 \mu_{ab}(s)
 :=\int_{\Hh\PP^2}t(x,\,z)^a t(y,\,z)^b\,d\mu(z)
 =\frac{a!b!}{(6)_{a+b}}
 [U^aV^b]\bigl(1-U-V+(1-s)UV\bigr)^{-2}.
\end{equation}
Consequently, if~$a+b\le5$, then
\begin{equation}\label{eq:mixed-design-moments}
 \sum_{z\in Y}t(x,\,z)^a t(y,\,z)^b
 =315\,\mu_{ab}(s).
\end{equation}
\end{lemma}

\begin{proof}
Choose unit representatives so that~$x$ is the first quaternionic coordinate line and~$y$ is represented by
\[
 \sqrt{s}\,e_1+\sqrt{1-s}\,e_2.
\]
On the underlying real~$12$-space, let~$A$ and~$B$ be the rank-$4$ orthogonal projectors onto these two quaternionic lines, and let~$G$ be a standard real Gaussian vector.  The Gaussian quadratic-form identity gives
\[
 \mathbb E\exp\!\left(\frac U2G^T A G+\frac V2G^T B G\right)
 =\det(I-UA-VB)^{-1/2}.
\]
The pair~$(A,\,B)$ decomposes into four identical real two-dimensional blocks.  On each block the determinant is
\[
 1-U-V+(1-s)UV,
\]
and the remaining real four-space contributes~$1$.  Hence the generating function is the inverse square occurring in \eqref{eq:two-point-haar-coefficient}.  Writing $G=RZ$, with~$Z$ uniform on the unit sphere and~$R$ independent, and using
\[
 \mathbb E(R^{2d})=2^d(6)_d,
\]
gives \eqref{eq:two-point-haar-coefficient}.  The function of~$z$ in \eqref{eq:mixed-design-moments} has projective degree~$a+b$, so the projective~$5$-design property gives the final assertion.
\end{proof}

\begin{lemma}[the rational moment kernel]\label{lem:rational-moment-kernel}
Let $Q=(q_{ij})\in M_5(\Q)$ and suppose
\begin{equation}\label{eq:homogeneous-mixed-moments}
 \sum_{i,\,j=1}^5q_{ij}\rho_i^a\rho_j^b=0
 \qquad(a,\,b\ge0,\,a+b\le5).
\end{equation}
Then
\begin{equation}\label{eq:moment-kernel-matrix}
 Q=\lambda\Omega,
 \qquad
 \Omega=uu^T,
 \qquad
 u=\begin{pmatrix}1&2&-1&-1&-1\end{pmatrix}^{T},
\end{equation}
for some~$\lambda\in\Q$.
\end{lemma}

\begin{proof}
Put $v_a=(\rho_i^a)_{i=1}^5\in K^5$ for~$a\ge0$,
with $v_0=(1,\,1,\,1,\,1,\,1)^T$.  The vectors~$v_0$, $\ldots$, $v_4$ form a
Vandermonde basis.  The equations with one exponent equal to~$0$ or~$1$
therefore give
\[
 Qv_0=Qv_1=0,
 \qquad v_0^TQ=v_1^TQ=0.
\]
Over~$K$, the remaining equations are the three conditions on the
pairs~$(v_2,\,v_2)$, $(v_2,\,v_3)$, and~$(v_3,\,v_2)$.  Thus without the
rationality assumption they leave a six-dimensional space of bilinear
forms on the three-dimensional quotient by~$\Span_{K}\{v_0,\,v_1\}$.

Since~$Q$ is rational, it also annihilates~$v_1'$ on both
sides, where the prime acts entrywise by the nontrivial automorphism
of~$K/\Q$.
Moreover,
\[
 \det[v_0,\,v_1,\,v_1',\,v_2,\,v_3]
 =-\frac{3\sqrt5}{8192}\ne0.
\]
Consequently~$Q$ descends to a bilinear form on a two-dimensional
quotient with basis represented by~$v_2$ and~$v_3$.  The three remaining
conditions force its first row and first column to vanish.  Hence~$Q$
annihilates
\[
 \Span_{K}\{v_0,\,v_1,\,v_1',\,v_2\}
\]
on both sides.  The common annihilator is spanned by $u=(1,\,2,\,-1,\,-1,\,-1)^T$, so $Q=\lambda uu^T$; its~$(1,\,1)$ entry shows
that~$\lambda\in\Q$.

Conversely, $u^Tv_a=0$ for~$0\le a\le2$, as follows from $g_++g_-=3/4$ and $g_+g_-=1/16$.  If~$a+b\le5$, then~$a\le2$ or~$b\le2$, so~$uu^T$ satisfies every equation in
\eqref{eq:homogeneous-mixed-moments}.
\end{proof}

\begin{theorem}[the forced five-angle association scheme]\label{thm:angle-association-scheme}
The five angle relations on~$Y$ form a symmetric five-class association scheme.  In the order \eqref{eq:angle-nodes}, the intersection matrix
\[
 P^{(k)}=(p_{ij}^{\,k})_{1\le i,\,j\le5}
\]
for a pair in relation~$\rho_k$ is as follows:
\begin{align*}
P^{(1)}&=
\begin{pmatrix}
1&8&0&0&0\\
8&8&64&0&0\\
0&64&64&16&16\\
0&0&16&0&16\\
0&0&16&16&0
\end{pmatrix},
&
P^{(2)}&=
\begin{pmatrix}
1&1&8&0&0\\
1&30&32&8&8\\
8&32&88&16&16\\
0&8&16&8&0\\
0&8&16&0&8
\end{pmatrix},\,\displaybreak[1]\\[2ex]
P^{(3)}&=
\begin{pmatrix}
0&4&4&1&1\\
4&16&44&8&8\\
4&44&77&17&17\\
1&8&17&1&5\\
1&8&17&5&1
\end{pmatrix},
&
P^{(4)}&=
\begin{pmatrix}
0&0&5&0&5\\
0&20&40&20&0\\
5&40&85&5&25\\
0&20&5&5&1\\
5&0&25&1&1
\end{pmatrix},\,\displaybreak[1]\\[2ex]
P^{(5)}&=
\begin{pmatrix}
0&0&5&5&0\\
0&20&40&0&20\\
5&40&85&25&5\\
5&0&25&1&1\\
0&20&5&1&5
\end{pmatrix}.
\end{align*}
In particular, all intersection numbers are forced by the projective~$5$-design identities and the five angle values.
\end{theorem}

\begin{proof}
Fix~$x\ne y$ with $t(x,\,y)=s=\rho_k$.  The terms $z=x$ and $z=y$ in \eqref{eq:mixed-design-moments} contribute~$s^b$ and~$s^a$, respectively.  Hence \eqref{eq:angle-intersection-numbers} satisfies
\begin{equation}\label{eq:full-mixed-moment-system}
 \sum_{i,\,j=1}^5p_{ij}(x,\,y)\rho_i^a\rho_j^b
 =315\,\mu_{ab}(s)-s^a-s^b
 \qquad(a+b\le5).
\end{equation}
Substitution in \eqref{eq:two-point-haar-coefficient} shows that the displayed~$P^{(k)}$ is one solution.  This is a finite exact calculation in~$\Q(\sqrt5)$; the resulting matrices also satisfy the balance and associativity identities.  \Cref{lem:rational-moment-kernel} shows that every integral solution has the form
\begin{equation}\label{eq:affine-intersection-solution}
 (p_{ij}(x,\,y))=P^{(k)}+\lambda\Omega,
 \qquad \lambda\in\Z.
\end{equation}

Nonnegativity immediately gives $\lambda=0$ for $k=1$, $2$, $4$, and~$5$.  Indeed, for $k=1$ the~$(1,\,3)$ and~$(4,\,4)$ entries are~$-\lambda$ and~$\lambda$; for $k=2$ the~$(1,\,4)$ and~$(4,\,5)$ entries are~$-\lambda$ and~$\lambda$; and for $k=4$ and~$5$ the~$(1,\,1)$ entry is~$\lambda$ while respectively the~$(1,\,4)$ or~$(1,\,5)$ entry is~$-\lambda$.

For $k=3$, nonnegativity gives~$\lambda\in\{0,\,1\}$.  Suppose $\lambda=1$.  Then there is a point~$z\in Y$ orthogonal to both~$x$ and~$y$.  Applied to the orthogonal pair~$x$, $z$, however, the already determined matrix~$P^{(1)}$ has
\[
 p_{31}^{\,1}=0,
\]
so there is no point which is simultaneously in the~$1/4$-relation to~$x$ and orthogonal to~$z$.  This contradicts the existence of~$y$.  Thus $\lambda=0$ also for $k=3$.

The resulting matrices depend only on the relation of~$(x,\,y)$, which is precisely the association-scheme condition.  The angle relations themselves are symmetric because $t(x,\,y)=t(y,\,x)$.
\end{proof}

\begin{corollary}[first eigenmatrix]\label{cor:first-eigenmatrix}
Let $A_0=I$ and let~$A_i$ be the adjacency matrix of the relation~$\rho_i$
for~$1\le i\le5$.  With columns ordered by the diagonal relation followed
by the five nodes in \eqref{eq:angle-nodes}, the first eigenmatrix
$P=(P_{ci})_{0\le c,\,i\le5}$ and the multiplicities~$m_c$ of the
primitive idempotents~$E_c$ of the Bose--Mesner algebra are
\begin{equation}\label{eq:first-eigenmatrix}
\begin{array}{c|rrrrrr|r}
 &A_0&A_1&A_2&A_3&A_4&A_5&\text{mult.}\\ \hline
 E_0&1&10&80&160&32&32&1\\
 E_1&1&5&10&0&-8&-8&36\\
 E_2&1&3&-4&-8&4&4&90\\
 E_3&1&-2&-4&7&-1&-1&160\\
 E_4&1&-5&20&-20&2+6\sqrt5&2-6\sqrt5&14\\
 E_5&1&-5&20&-20&2-6\sqrt5&2+6\sqrt5&14
\end{array}
\end{equation}
The two multiplicity-$14$ eigenspaces are Galois conjugate and are the only eigenspaces which distinguish the two golden relations.  Moreover, at the chosen real embedding~$\sigma_1$, the centered Gram matrix
\[
 G^{(1)}_{xy}=4t(x,\,y)-\frac43
\]
is~$60E_4$ with the labeling above.  At the conjugate embedding~$\sigma_2$, the two golden values are exchanged and the corresponding centered Gram matrix~$G^{(2)}$ is~$60E_5$.  Thus the image at either real place is one of these~$14$-dimensional eigenspaces; this is also the dimension of the trace-zero Jordan space~$\Herm_3(\Hh)^0$ used in \cref{sec:geometric-rigidity}.
\end{corollary}

\begin{proof}
Simultaneous diagonalization of the commuting intersection matrices gives \eqref{eq:first-eigenmatrix}.  If $v_i=P_{0i}$ denotes the valency of~$A_i$, row orthogonality gives the checkable formula
\[
 m_c=\frac{|Y|}{\displaystyle\sum_{i=0}^5 P_{ci}^2/v_i},
\]
which yields the displayed multiplicities.  Substituting the six projective-inner-product values in the rows shows that~$G^{(1)}$ has eigenvalue~$60$ on~$E_4$ and eigenvalue~$0$ on the other five eigenspaces; Galois conjugation gives the corresponding statement for~$G^{(2)}$ and~$E_5$.
\end{proof}

\begin{remark}[why the scheme theorem is special here]\label{rem:beyond-general-design-scheme-criterion}
The standard sufficient criterion that an~$s$-distance design of strength~$\ell$ carry an association scheme requires~$\ell\ge2s-2$ \cite{DGS,HoggarDesigns}.  Here $s=\ell=5$, so the criterion would require strength~$8$ and does not apply.  The five-class scheme is instead forced by the special arithmetic angle set, the rationality and integrality of the intersection numbers, and nonnegativity.
\end{remark}

\begin{corollary}[distance fusion]\label{prop:distance-angle}
The orthogonality graph~$\Gamma$ on~$Y$ is distance-regular of diameter~$4$.  Its distance matrices are
\[
 D_0=A_0,
 \qquad D_1=A_1,
 \qquad D_2=A_2,
 \qquad D_3=A_3,
 \qquad D_4=A_4+A_5,
\]
and its intersection array is
\begin{equation}\label{eq:intersection-array}
 \{10,\,8,\,8,\,2;1,\,1,\,4,\,5\}.
\end{equation}
Equivalently,
\begin{equation}\label{eq:angle-distance-table}
\begin{array}{c|cccccc}
 t(x,\,y)&1&0&\frac12&\frac14&g_+&g_-\\ \hline
 d_\Gamma(x,\,y)&0&1&2&3&4&4.
\end{array}
\end{equation}
\end{corollary}

\begin{proof}
The intersection tables give
\begin{align*}
 A_1^2&=10A_0+A_1+A_2,\\
 A_1A_2&=8A_1+A_2+4A_3,\\
 A_1A_3&=8A_2+4A_3+5(A_4+A_5),\\
 A_1(A_4+A_5)&=2A_3+5(A_4+A_5).
\end{align*}
These are the distance-polynomial recurrences for \eqref{eq:intersection-array}.  In particular, they show successively that~$A_2$, $A_3$, and~$A_4+A_5$ are the distance-$2$, distance-$3$, and distance-$4$ matrices.  For the same graph, intersection array, and spectrum in the standard distance-regular-graph tables, cf.\ \cite[pp.~408--410]{BCN}.
\end{proof}

\begin{remark}[logical direction]\label{rem:angle-moment-direction}
For the norm-$2$ shell of any extremal~$L$, \cref{prop:extremal-shell-angles} derives the five-value restriction from icosian integrality, extremality, and reduction modulo~$2$, while \cref{lem:primitive-shell-lines} proves multiplicity-freeness and \cref{lem:norm2-multiplicity-free} determines the valencies from projective moments.  No coordinate orbit calculation is used.  For a general comparison set~$Y$ of cardinality~$315$, the same
five-angle set is an explicit hypothesis, and equality in
\cref{lem:hoggar-bound-equality} supplies the design property.  The
valencies, all intersection numbers, the distance fusion, and the
intersection array then follow from the projective~$5$-design identities.  Cohen--Tits enters only when comparing the resulting fused geometries.
\end{remark}

\subsection{Orthogonal frames and the near-octagon geometry}

The identity
\[
 A_1^2=10A_0+A_1+A_2
\]
shows that every edge of~$\Gamma$ lies in a unique triangle.  Three pairwise orthogonal quaternionic lines in~$\Hh^3$ form a complete orthogonal frame.  Each point has ten neighbors, paired into five triangles through the point, and hence the resulting incidence geometry~$\mathcal N_Y$ has
\[
 \frac{315\cdot5}{3}=525
\]
lines.  For $Y=X_L$, these triangles are precisely the orthogonal norm-$2$ frames of~$L$.  Write~$\Fr(L)$ for the set of these unordered frames.  Thus every extremal class has
\[
 |\Fr(L)|=525;
\]
this is the count used later in the mass--code argument.

\begin{theorem}[near-octagon construction]\label{thm:near-octagon}
The incidence geometry~$\mathcal N_Y$, whose points are~$Y$ and whose lines are the orthogonal triples, is a near octagon of order~$(2,\,4)$.  Its point graph is~$\Gamma$, and it has the regular parameters customarily denoted~$(2,\,4;0,\,3)$.
\end{theorem}

\begin{proof}
Let $\ell=\{p_1,\,p_2,\,p_3\}$ be a frame, where the same symbols denote the rank-one projectors.  Then
\[
 p_1+p_2+p_3=I_3.
\]
For every~$x\in Y$,
\begin{equation}\label{eq:frame-inner-product-sum}
 t(x,\,p_1)+t(x,\,p_2)+t(x,\,p_3)=1.
\end{equation}
The only unordered triples chosen from
\[
 \left\{0,\,\frac14,\,\frac12,\,g_+,\,g_-,\,1\right\}
\]
whose sum is~$1$ are
\begin{equation}\label{eq:frame-angle-triples}
 (1,\,0,\,0),\qquad
 \left(0,\,\frac12,\,\frac12\right),\qquad
 \left(\frac12,\,\frac14,\,\frac14\right),\qquad
 \left(\frac14,\,g_+,\,g_-\right).
\end{equation}
Using \eqref{eq:angle-distance-table}, these correspond respectively to distance patterns
\[
 (0,\,1,\,1),\qquad(1,\,2,\,2),\qquad(2,\,3,\,3),\qquad(3,\,4,\,4)
\]
from~$x$ to the three points of~$\ell$.  Thus every point has a unique nearest point on every line.  The point graph has diameter~$4$, so~$\mathcal N_Y$ is a near octagon.  Every line has three points and every point lies on five lines, so the order is~$(2,\,4)$.  For two points at distance $i=2$ or~$3$, the number~$t_i+1$ of lines through the second point containing a point one step nearer the first equals the backward intersection number~$c_i$.  Hence \eqref{eq:intersection-array} gives $t_2=c_2-1=0$ and $t_3=c_3-1=3$, proving the asserted regular parameters.
\end{proof}

\subsection{Uniqueness of the golden-angle fission}

\begin{proposition}[uniqueness of the golden fission]\label{prop:golden-fission-uniqueness}
Let~$\Gamma_0$ be the point graph of a regular near octagon with parameters~$(2,\,4;0,\,3)$, and let~$D_i$ denote its distance-$i$ relation.  Suppose that
\[
 D_4=R_4\mathbin{\dot\cup}R_5
\]
is a symmetric fission with both valencies equal to~$32$ for which
\[
 D_0,\,D_1,\,D_2,\,D_3,\,R_4,\,R_5
\]
have the intersection tables displayed in \cref{thm:angle-association-scheme}.  Then there is at most one such unordered pair~$\{R_4,\,R_5\}$; equivalently, any two such fissions differ at most by the global exchange~$R_4\leftrightarrow R_5$.
\end{proposition}

\begin{proof}
The point graph~$\Gamma_0$ has intersection array
\[
 \{10,\,8,\,8,\,2;1,\,1,\,4,\,5\}.
\]
Fix a point~$x$ and write~$\Gamma_{0,\,4}(x)$ for the set of points
at distance~$4$ from~$x$.  Let~$H_x$ be the graph induced there by~$\Gamma_0$.
Each vertex has~$10-c_4=5$ neighbors in this layer, so~$H_x$ is~$5$-regular.  Put
\[
 R_i(x)=\{y:(x,\,y)\in R_i\},\qquad i=4,\,5.
\]
For~$y\in R_4(x)$, the entries
\[
 p_{41}^{\,4}=0,\qquad p_{51}^{\,4}=5
\]
show that all five neighbors of~$y$ in~$H_x$ lie in~$R_5(x)$; for~$y\in R_5(x)$ the entries
\[
 p_{41}^{\,5}=5,\qquad p_{51}^{\,5}=0
\]
give the reverse statement.  Thus~$H_x$ is bipartite with parts~$R_4(x)$ and~$R_5(x)$, each of size~$32$.

The graph~$H_x$ is connected.  Indeed, let~$y$ be a vertex.  Its five neighbors in~$H_x$ each have four neighbors other than~$y$.  The resulting twenty vertices are distinct.  For if two different two-step paths~$y-u-z$ and~$y-v-z$ had the same endpoint, then~$y$ and~$z$ would be distinct nonadjacent vertices with the two common neighbors~$u$ and~$v$.  They would therefore be at distance~$2$ in~$\Gamma_0$, contradicting $c_2=1$.  Hence the component of~$y$ contains at least~$21$ vertices in the same bipartition class as~$y$.  Every component of a~$5$-regular bipartite graph has equally many vertices in its two parts, so every component of~$H_x$ contains at least~$21$ vertices in each part.  Since each whole part has only~$32$ vertices, $H_x$ is connected.  Its bipartition is therefore unique up to exchange.

Now let~$(R_4,\,R_5)$ and~$(S_4,\,S_5)$ be two fissions with the stated parameters.  For pairs at distance~$4$, write
\[
 \sigma_R(x,\,y)=
 \begin{cases}
  1,&(x,\,y)\in R_4,\\
 -1,&(x,\,y)\in R_5,
 \end{cases}
 \qquad
 \sigma_S(x,\,y)=
 \begin{cases}
  1,&(x,\,y)\in S_4,\\
 -1,&(x,\,y)\in S_5.
 \end{cases}
\]
The uniqueness of the bipartition of~$H_x$ gives a sign~$\varepsilon_x\in\{\pm1\}$ such that
\[
 \sigma_S(x,\,y)=\varepsilon_x\sigma_R(x,\,y)
 \qquad\bigl(y\in\Gamma_{0,\,4}(x)\bigr).
\]
Since both fissions are symmetric, for $d_{\Gamma_0}(x,\,y)=4$ one has
\[
 \varepsilon_x\sigma_R(x,\,y)
 =\sigma_S(x,\,y)
 =\sigma_S(y,\,x)
 =\varepsilon_y\sigma_R(y,\,x),
\]
and hence $\varepsilon_x=\varepsilon_y$.

Finally, the distance-$4$ graph of~$\Gamma_0$ is connected.  For adjacent~$x$ and~$y$, the table~$P^{(1)}$ gives
\[
 \#\bigl(\Gamma_{0,\,4}(x)\cap\Gamma_{0,\,4}(y)\bigr)
 =\sum_{i,\,j=4}^{5}p_{ij}^{\,1}=32.
\]
Thus every edge of the connected graph~$\Gamma_0$ can be replaced by a path of length two in its distance-$4$ graph.  Consequently all~$\varepsilon_x$ are equal, and the two fissions agree globally or are globally exchanged.
\end{proof}

\section{The Hall--Janko lattice: automorphisms and uniqueness}\label{sec:tits-lattice}

\subsection{Tits's lattice and the residual mass}

Let~$\mathfrak P_0$ and~$\mathfrak P$ be two distinct maximal right ideals of~$\cI$ above~$2$.  Tits's lattice may be written
\begin{equation}\label{eq:tits-model}
 \HJ=
 \left\{(\lambda_1,\,\lambda_2,\,\lambda_3)\in\cI^3:
 \lambda_j\equiv\lambda_1\pmod{\mathfrak P},
 \quad
 \lambda_1+\lambda_2+\lambda_3\equiv0\pmod{\mathfrak P_0}
 \right\},
\end{equation}
with Hermitian form
\begin{equation}\label{eq:tits-form}
 h_{\mathrm T}(x,\,y)=\frac12\sum_{j=1}^3\overline{x_j}y_j.
\end{equation}
Taking $\pi=1+i$, of reduced norm~$2$, coordinatewise left
multiplication by~$\pi^{-1}$ identifies~$(D^3,\,h_{\mathrm T})$ with the
standard Hermitian space.  All norms, shells, and isometries of this
congruence model are taken with~$h_{\mathrm T}$.  We verify unimodularity
and extremality directly from the congruences.

Under~$\cI/2\cI\cong M_2(\F_4)$, a right ideal consists of matrices
whose columns lie in a fixed line.  Choose coordinates in which the
lines for~$\mathfrak P$ and~$\mathfrak P_0$ are the two coordinate
lines.  Each column of a vector in \eqref{eq:tits-model} then lies in the code
\begin{equation}\label{eq:frame-code}
 C=
 \left\{
 ((x_1,\,b),\,(x_2,\,b),\,(x_3,\,b))\in(\F_4^2)^3:
 x_1+x_2+x_3=0
 \right\}.
\end{equation}
Canonical quaternionic conjugation reduces
to the adjugate involution.  Under Morita equivalence the integrality
and self-duality conditions for~$h_{\mathrm T}$ therefore become,
respectively, isotropy and self-duality for the standard symplectic
pairing on~$(\F_4^2)^3$, up to a nonzero scalar.  The distinct lines
form a hyperbolic pair, and on the displayed code the pairing is
\[
 \bigl\langle ((x_j,\,b)),\,((y_j,\,c))\bigr\rangle
 =c\sum_jx_j+b\sum_jy_j=0.
\]
Since $\dim_{\F_4}C=3$, it is maximal isotropic in a six-dimensional
symplectic space.  More explicitly, $2\cI^3$ and~$\cI^3$ are dual
for~$h_{\mathrm T}$ at~$2$, and the intermediate lattice is self-dual
exactly when its Morita code is self-dual.  At every other finite place
the congruences impose no restriction and~$1/2$ is a unit.  Hence~$\HJ$ is integral and unimodular.

The two distinct maximal right ideals of~$M_2(\F_4)$ meet trivially,
so
\begin{equation}\label{eq:tits-separated-lines}
 \mathfrak P\cap\mathfrak P_0=2\cI,
 \qquad \HJ\cap De_j=2\cI e_j\quad(1\le j\le3).
\end{equation}
If $h_{\mathrm T}(\lambda,\,\lambda)=1$, then $\sum_j\nrd(\lambda_j)=2$.  Every nonzero summand is a totally
positive integer of~$K$ with both conjugates at most~$2$.  Writing it
as~$m+n\tau$ gives~$|n|\sqrt5<2$, hence $n=0$; the only possibilities
are~$1$ and~$2$.  If two coordinates have norm~$1$ and the third is
zero, the congruences force both nonzero coordinates into~$\mathfrak P$, although their reductions have determinant~$1$.
If only one coordinate is nonzero, it has norm~$2$ and belongs to~$2\cI$ by \eqref{eq:tits-separated-lines}, which is impossible because
its reduced norm would lie in~$4\cO_K$.  Thus
\begin{equation}\label{eq:tits-extremal}
 r_{\HJ}(1)=0,
\end{equation}
and hence also $r_{\HJ}(\tau^2)=0$ by \eqref{eq:central-unit-shell-symmetry}.  Thus~$\HJ$ is a nonsplit class in~$\cG_3$.

The mass left after removing the split class is
\begin{equation}\label{eq:nonsplit-mass}
 \Mass_{\mathrm{ns}}
 =\Mass(\cG_3)-\frac1{|\Aut_{\cI}(\Ls)|}
 =\frac{67-7}{72\,576\,000}
 =\frac{60}{72\,576\,000}
 =\frac1{1\,209\,600}.
\end{equation}
Since~$\HJ$ is one of the nonsplit classes, its mass contribution is at most the total residual mass:
\[
 \frac1{|\Aut_{\cI}(\HJ)|}
 \le \Mass_{\mathrm{ns}}.
\]
Therefore
\begin{equation}\label{eq:aut-tits-lower}
 |\Aut_{\cI}(\HJ)|\ge1\,209\,600.
\end{equation}
Equality holds precisely when~$\HJ$ exhausts the residual mass, equivalently when it is the unique nonsplit class.  Thus the automorphism-group order and the class-number statement will follow simultaneously from the opposite upper bound.

\subsection{The intrinsic frames, the framed code, and mass saturation}\label{sec:frames}

Let
\[
 X=X_{L_{\mathrm T}}=(\HJ)_2/\cI^1.
\]
Since~$\HJ$ is extremal, \cref{lem:norm2-multiplicity-free,thm:near-octagon} show that~$X$ is a~$315$-point projective~$5$-design with exactly
\begin{equation}\label{eq:frame-count}
 |\Fr(\HJ)|=525
\end{equation}
orthogonal frames.  Thus the frame count used below is a consequence of the projective moments, not an imported coordinate-orbit calculation.

Fix the frame~$F_0\in\Fr(\HJ)$ formed by the three coordinate lines in \eqref{eq:tits-model}.  Its separated sublattice is
\[
 M_{F_0}=2\cI\oplus2\cI\oplus2\cI.
\]
This description follows from \eqref{eq:tits-separated-lines}.
The reduction is a right~$M_2(\F_4)$-module: right multiplication mixes the two columns, whereas~$\cI$-linear block isometries act on each column by left matrix
multiplication.  Taking one column is the Morita equivalence for these
right modules.  The congruences in \eqref{eq:tits-model} act
columnwise, so~$\HJ/M_{F_0}$ corresponds to two copies of the same
one-column code.  Let~$\ell$ and~$\ell_0$ be the distinct lines
corresponding to~$\mathfrak P/2\cI$ and~$\mathfrak P_0/2\cI$.
Before choosing coordinates, the code is
\begin{equation}\label{eq:frame-code-invariant}
 \left\{(v_1,\,v_2,\,v_3)\in(\F_4^2)^3:
 v_j-v_1\in\ell,\quad
 v_1+v_2+v_3\in\ell_0\right\}.
\end{equation}  We use the column-vector Morita functor: left multiplication by a norm-one residue class acts by the usual left action of~$\SL_2(4)$ on~$\F_4^2$.  Choose
\[
 \ell=\F_4(1,\,0),
 \qquad
 \ell_0=\F_4(0,\,1).
\]
Writing $v_j=(x_j,\,b_j)$, the first condition gives $b_1=b_2=b_3$, while the second gives $x_1+x_2+x_3=0$.
Thus this is precisely the code~$C$ in \eqref{eq:frame-code}.
The full quotient is~$C^{\oplus2}$; indeed $|C|=4^3=64$ and $|C^{\oplus2}|=4096=|\HJ/M_{F_0}|$.  Since left multiplication acts by the same matrix on both columns, stabilizing~$C^{\oplus2}$ is equivalent to stabilizing~$C$.  No Frobenius-semilinear automorphisms of~$\F_4$ occur, because the reduction of~$\cI^1$ lies in the linear group~$\SL_2(4)$.

\begin{proposition}[frame stabilizer]\label{prop:code-stabilizer}
The stabilizer of~$F_0$ in~$\Aut_{\cI}(\HJ)$ has order at most~$2304$.
\end{proposition}

\begin{proof}
First consider the stabilizer of~$C$ in~$\SL_2(4)^3\rtimes S_3$.  If
\[
 g_i=\begin{pmatrix}a_i&c_i\\ b_i&d_i\end{pmatrix}\in\SL_2(4)
\]
fixes the three blocks, preservation of \eqref{eq:frame-code} forces
\[
 b_1=b_2=b_3=0,
 \qquad
 d_1=d_2=d_3=s,
 \qquad
 a_i=s^{-1},
 \qquad
 c_1+c_2+c_3=0.
\]
There are $3\cdot4^2=48$ such triples.  The block-fixing subgroup
therefore has structure
\[
 \F_4^2\rtimes\F_4^\times\cong2^4\!:\!3,
\]
and adjoining the six block permutations gives the full code stabilizer
\[
 2^4\!:\!(3\times S_3)
\]
of order~$288$.

An~$\cI$-linear isometry preserving~$F_0$ is monomial.  By
\eqref{eq:tits-separated-lines}, it permutes the lattices~$2\cI e_j$,
so its nonzero entries are units of~$\cI$; the isometry condition puts
them in~$\cI^1$.  It therefore preserves~$M_{F_0}$.  Reduction modulo~$2\cI$ maps its stabilizer to the code stabilizer, and \eqref{eq:reduction-mod2} shows that the kernel is contained in the independent sign changes on the three frame lines.  Hence
\[
 |\Aut_{\cI}(\HJ)_{F_0}|\le2^3\cdot288=2304.
\]
\end{proof}

The full automorphism group permutes~$\Fr(\HJ)$, so orbit--stabilizer and \eqref{eq:frame-count} give
\begin{equation}\label{eq:aut-tits-upper}
 |\Aut_{\cI}(\HJ)|
 \le525\cdot2304
 =1\,209\,600.
\end{equation}
Combining \eqref{eq:aut-tits-lower} and \eqref{eq:aut-tits-upper} yields
\begin{equation}\label{eq:tits-aut-order}
 \boxed{|\Aut_{\cI}(\HJ)|=1\,209\,600.}
\end{equation}
Consequently the contribution of~$\HJ$ is the entire residual mass \eqref{eq:nonsplit-mass}.  Positivity of the remaining mass summands gives
\[
 \cG_3=\{[\Ls],\,[\HJ]\}.
\]
Thus~$\HJ$ is the unique nonsplit, and hence the unique extremal, class.  Both the order computation and the uniqueness statement have been obtained without using the order of a previously identified finite group.

\begin{remark}[mass-assisted frame method]\label{rem:mass-assisted-frame-method}
The finite code supplies only the uniform upper bound $|\Aut_{\cI}(\HJ)_{F_0}|\le2304$; no transitivity on frames is assumed.  The residual mass supplies the opposite global inequality, and the identity
\[
 \frac1{10\,368\,000}+\frac1{1\,209\,600}
 =\frac{67}{72\,576\,000}
\]
forces equality throughout.  Thus mass saturation simultaneously proves the class-number statement, sharpness of the stabilizer bound, and transitivity on the~$525$ frames.  This is a mass-assisted form of Conway's frame method.
\end{remark}

\begin{corollary}[Tits's reflection identification]\label{cor:tits-reflection-group}
The group~$\Aut_{\cI}(\HJ)$ is generated by quaternionic reflections in its norm-$2$ lines and is isomorphic to~$2.J_2$.
\end{corollary}

\begin{proof}
For a norm-$2$ vector~$x$, put $\ell=xD$ and
\begin{equation}\label{eq:reflection}
 s_\ell(v)=v-2x h_{\mathrm T}(x,\,x)^{-1}h_{\mathrm T}(x,\,v).
\end{equation}
The formula depends only on~$\ell$.
Since $h_{\mathrm T}(x,\,x)=2$, it simplifies to
$s_\ell(v)=v-xh_{\mathrm T}(x,\,v)$.
Integrality gives~$h_{\mathrm T}(x,\,v)\in\cI$ for~$v\in\HJ$, so~$s_\ell$ preserves~$\HJ$; it is an~$\cI$-linear isometric involution.  This argument works
for every norm-$2$ vector in an integral Hermitian~$\cI$-lattice.
The external input is Tits's identification of the generated
reflection group with~$2.J_2$ \cite{Tits}.  Its order is~$1\,209\,600$,
which agrees with the independently determined order in
\eqref{eq:tits-aut-order}; hence the reflection group is all of~$\Aut_{\cI}(\HJ)$.
\end{proof}

Together with \eqref{eq:split-aut}, this proves \cref{thm:main}.

\begin{proposition}[Galois exchange of the golden relations]\label{prop:golden-galois-symmetry}
There is a permutation~$\iota$ of~$X$ which fixes the relations with projective inner products~$0$, $1/2$, and~$1/4$, and interchanges the two relations with projective inner products~$g_+$ and~$g_-$.
\end{proposition}

\begin{proof}
The automorphism~$\gamma$ of \eqref{eq:icosian-galois-map} permutes
the five maximal right ideals above~$2$.  Its coordinatewise image
of \eqref{eq:tits-model} is the same congruence model for the ordered
pair $(\gamma(\mathfrak P),\,\gamma(\mathfrak P_0))$.
The group~$\SL_2(4)$ is transitive on ordered pairs of distinct
lines in~$\F_4^2$.  By \eqref{eq:reduction-mod2}, choose~$u\in\cI^1$ with $u\gamma(\mathfrak P)=\mathfrak P$ and $u\gamma(\mathfrak P_0)=\mathfrak P_0$.
Then
\[
 \eta(\lambda_1,\,\lambda_2,\,\lambda_3)
   =(u\gamma(\lambda_1),\,u\gamma(\lambda_2),\,u\gamma(\lambda_3))
\]
preserves~$\HJ$: using the same unit in all coordinates preserves
both the difference and sum congruences.  Moreover, $h_{\mathrm T}(\eta x,\,\eta y)=\gamma(h_{\mathrm T}(x,\,y))$.
Thus~$\eta$ induces $\iota([x])=[\eta(x)]$ on~$X$, and the
reduced norms of inner products are Galois-conjugated.  This fixes
the rational angles and exchanges~$g_+$ and~$g_-$.
\end{proof}

\begin{remark}[point transitivity]\label{rem:point-transitivity}
The projective action of~$\Aut_{\cI}(\HJ)$ on~$X$ is transitive, but this is not used in the classification or rigidity proof.  Equality in the frame-stabilizer bound implies transitivity on the~$525$ frames and gives the full~$S_3$ action on the three lines of a frame.  Since every point lies on five frames, flag counting gives one point orbit of size $525\cdot3/5=315$.
\end{remark}

\begin{remark}[Comparison with the known class-number computations]\label{rem:classification-comparison}
Coulangeon proves the same two-class result by local neighbors: $\cI^3$ has a unique irreducible~$2\cI$-neighbor, identified with the Tits lattice, and a prime above~$5$ excludes a further rootless class \cite{Coulangeon95}.  Kirschmer independently records class number~$2$, with group labels~$[\SL_2(5)]_1^3$ and~$[2.J_2]_3$ \cite[Theorem~9.3.2 and p.~150]{Kirschmer}.  The present proof instead derives the nonsplit shell counts from the Hilbert Eisenstein average, bounds the automorphism group through the~$525$ frames and their finite shadow, and closes the argument by residual-mass saturation.
\end{remark}

\section{Combinatorial and geometric rigidity}\label{sec:geometric-rigidity}

Equality in Hoggar's bound supplies the design property, and the
intrinsic shell calculation determines the association-scheme
parameters for every~$315$-point comparison set with the prescribed
five projective inner products.  The Tits congruence model supplies
the arithmetic symmetry that exchanges the two golden labels.

\begin{corollary}[combinatorial rigidity for the prescribed angles]\label{cor:prescribed-angle-scheme-uniqueness}
Let $Y\subset\Hh\PP^2$ be a set of~$315$ points whose projective inner
products between distinct points belong to \eqref{eq:five-angle-set}.
Then there is a bijection
\[
 \phi:Y\longrightarrow X
\]
preserving all projective inner products.
\end{corollary}

\begin{proof}
By \cref{lem:hoggar-bound-equality}, $Y$ is a projective~$5$-design.
Thus \cref{thm:angle-association-scheme} gives the displayed five-class intersection numbers for~$Y$, and \cref{thm:near-octagon} shows that its golden-angle fusion is a regular near octagon with parameters~$(2,\,4;0,\,3)$.  By Cohen--Tits uniqueness \cite{CohenTits}, this fused geometry is isomorphic to the Hall--Janko near octagon.  After transporting the two golden relations of~$Y$ along such an isomorphism, \cref{prop:golden-fission-uniqueness} shows that they are the two golden relations of~$X$, either in the corresponding or in the exchanged order.  In the latter case, compose with the Galois permutation~$\iota$ from \cref{prop:golden-galois-symmetry}.  The resulting bijection preserves every projective inner product.
\end{proof}

\subsection{The cubic moment and quaternionic rigidity}

In~$J=\Herm_3(\Hh)$, with identity~$I_3$, let
\[
 J_0=\{a\in J:\Tr_J(a)=0\}
\]
be the trace-zero subspace.  The Jordan product is
$a\circ b=(ab+ba)/2$, with positive trace form
$T(a,\,b)=\Tr_J(a\circ b)$ as in \cref{sec:projective-harmonics}.
The rank-one projectors representing~$\Hh\PP^2$ are precisely the primitive
idempotents~$p\in J$ of trace~$1$, and $T(p,\,q)=t(p,\,q)$.
Center the projective plane in the~$14$-dimensional space~$J_0$ by
\begin{equation}\label{eq:quaternionic-centering}
 \widehat p=2p-\frac23I_3.
\end{equation}
Then
\begin{equation}\label{eq:quaternionic-centered-angle}
 T(\widehat p,\,\widehat q)=4t(p,\,q)-\frac43,
 \qquad
 T(\widehat p,\,\widehat p)=\frac83.
\end{equation}

\begin{lemma}[invariants of the trace-zero quaternionic Jordan module]\label{lem:quaternionic-jordan-invariants}
The invariant polynomial algebra on~$J_0$ is
\begin{equation}\label{eq:quaternionic-invariant-ring}
 \R[J_0]^{P\Sp(3)}
 =\R[q,\,c],
 \qquad
 q(a)=T(a,\,a),
 \qquad
 c(a)=\Tr_J(a^3).
\end{equation}
In particular, invariant quadratic and cubic polynomials are scalar multiples of~$q$ and~$c$, respectively.
\end{lemma}

\begin{proof}
By the spectral theorem for the Euclidean Jordan algebra~$J$, every~$a\in J_0$ is conjugate under $P\Sp(3)=\Aut(J)$ to
\[
 a=\lambda_1p_1+\lambda_2p_2+\lambda_3p_3,
 \qquad
 \lambda_1+\lambda_2+\lambda_3=0,
\]
for a Jordan frame~$(p_1,\,p_2,\,p_3)$.  The group is transitive on frames and the frame stabilizer induces the full symmetric group on the three idempotents.  Restriction to the diagonal plane therefore identifies the invariant algebra with the symmetric polynomials in~$\lambda_1$, $\lambda_2$, and~$\lambda_3$ subject to their sum being zero.  This algebra is generated by
\[
 \lambda_1^2+\lambda_2^2+\lambda_3^2=q(a)
 \quad\text{and}\quad
 \lambda_1^3+\lambda_2^3+\lambda_3^3=c(a).
\]
See also \cite[Chs.~5 and~7]{SpringerVeldkamp}.
\end{proof}

\begin{lemma}[the trace form and cubic as finite design moments]\label{lem:quaternionic-cubic-moment}
Let $Y\subset\Hh\PP^2$ be a finite quaternionic projective~$3$-design of cardinality~$N$.  Then, for~$a$ and~$b$ in~$J_0$,
\begin{align}
 \sum_{p\in Y}T(a,\,\widehat p)T(b,\,\widehat p)
 &=\frac{4N}{21}T(a,\,b),
 \label{eq:quaternionic-second-moment}\\
 \sum_{p\in Y}T(a,\,\widehat p)^3
 &=\frac{2N}{21}\Tr_J(a^3).
 \label{eq:quaternionic-third-moment}
\end{align}
In particular, for the Hall--Janko design~$X$ these constants are~$60$ and~$30$.
\end{lemma}

\begin{proof}
The design property replaces each finite average by the corresponding invariant integral over~$\Hh\PP^2$.  By \cref{lem:quaternionic-jordan-invariants},
\[
 \mathcal M_2(a)=\int_{\Hh\PP^2}T(a,\,\widehat p)^2\,d\mu(p)
\]
is a scalar multiple of~$q(a)$.  Taking the trace over an orthonormal basis of the~$14$-dimensional space~$J_0$ and using \eqref{eq:quaternionic-centered-angle} gives
\[
 \mathcal M_2(a)=\frac{8/3}{14}q(a)=\frac4{21}q(a).
\]
Polarization gives \eqref{eq:quaternionic-second-moment}.

Likewise,
\[
 \mathcal M_3(a)=\int_{\Hh\PP^2}T(a,\,\widehat p)^3\,d\mu(p)
 =\kappa\Tr_J(a^3)
\]
for some~$\kappa$.  Fix~$p_0\in\Hh\PP^2$ and put $a=\widehat p_0$.  By \eqref{eq:haar-moments} and \eqref{eq:quaternionic-centered-angle},
\begin{align*}
 \mathcal M_3(\widehat p_0)
 &=\int_{\Hh\PP^2}\left(4t(p_0,\,p)-\frac43\right)^3d\mu(p)\\
 &=64\frac{(2)_3}{(6)_3}
   -64\frac{(2)_2}{(6)_2}
   +\frac{64}{3}\frac{(2)_1}{(6)_1}
   -\frac{64}{27}
 =\frac{32}{189}.
\end{align*}
The eigenvalues of~$\widehat p_0$ are~$4/3$, $-2/3$, and~$-2/3$, so
\[
 \Tr_J(\widehat p_0^3)=\frac{16}{9}.
\]
Consequently $\kappa=2/21$, proving \eqref{eq:quaternionic-third-moment}.
\end{proof}

\begin{theorem}[quaternionic cubic rigidity]\label{thm:quaternionic-cubic-rigidity}
Let~$Y$ and~$Y'$ be finite quaternionic projective~$3$-designs in~$\Hh\PP^2$, and let
\[
 \phi:Y\longrightarrow Y'
\]
be a bijection satisfying
\[
 t(\phi(p),\,\phi(q))=t(p,\,q)
 \qquad(p,\,q\in Y).
\]
Then there is a unique
\[
 g\in P\Sp(3)=\Aut(J)
\]
such that $g(p)=\phi(p)$ for every~$p\in Y$.
\end{theorem}

\begin{proof}
By \eqref{eq:quaternionic-centered-angle}, the bijection
\[
 \widehat p\longmapsto\widehat{\phi(p)}
\]
preserves the Gram matrix.  Equation~\eqref{eq:quaternionic-second-moment} implies that both centered configurations span~$J_0$, and hence the bijection extends uniquely to an orthogonal transformation
\[
 g_0\in O(J_0,\,T).
\]
Using \eqref{eq:quaternionic-third-moment} for~$Y$ and~$Y'$ gives, for every~$a\in J_0$,
\begin{align*}
 \frac{2|Y|}{21}\Tr_J((g_0a)^3)
 &=\sum_{p'\in Y'}T(g_0a,\,\widehat p')^3\\
 &=\sum_{p\in Y}T(g_0a,\,g_0\widehat p)^3\\
 &=\sum_{p\in Y}T(a,\,\widehat p)^3
 =\frac{2|Y|}{21}\Tr_J(a^3).
\end{align*}
Thus~$g_0$ preserves both the trace form and the cubic $c(a)=\Tr_J(a^3)$.

The symmetric trilinear form
\[
 C(a,\,b,\,c)=T(a\circ b,\,c),
 \qquad a,\,b,\,c\in J_0,
\]
is the polarization of~$a\mapsto\Tr_J(a^3)$.  Define the traceless Jordan product by
\[
 a*b=a\circ b-\frac13T(a,\,b)I_3.
\]
For~$c\in J_0$ one has $T(a*b,\,c)=C(a,\,b,\,c)$.  Preservation of~$T$ and~$C$, together with nondegeneracy of~$T$, therefore gives
\[
 g_0(a*b)=g_0(a)*g_0(b).
\]
Extend~$g_0$ to $J=\R I_3\oplus J_0$ by fixing~$I_3$.  For real
scalars~$s$ and~$t$,
\[
 (sI_3+a)\circ(tI_3+b)
 =\left(st+\frac13T(a,\,b)\right)I_3
  +sb+ta+a*b,
\]
so the extension is a Jordan automorphism, hence belongs to~$P\Sp(3)$.  It carries~$p$ to~$\phi(p)$ because it carries~$\widehat p$ to~$\widehat{\phi(p)}$.  Uniqueness follows from the spanning of the centered configuration.
\end{proof}

\begin{theorem}[maximality and geometric rigidity for the prescribed angles]\label{thm:HJ-prescribed-angle-uniqueness}
Let $Y\subset\Hh\PP^2$ have all projective inner products between
distinct points in \eqref{eq:five-angle-set}.  Then~$|Y|\le315$,
with equality if and only if~$Y$ is~$P\Sp(3)$-equivalent to the
Hall--Janko configuration~$X$.
\end{theorem}

\begin{proof}
The bound and the projective~$5$-design property at equality follow
from \cref{lem:hoggar-bound-equality}.  If $|Y|=315$,
\cref{cor:prescribed-angle-scheme-uniqueness} supplies a
bijection~$Y\to X$ preserving all projective inner products.  Both sets
are projective~$3$-designs, so \cref{thm:quaternionic-cubic-rigidity}
shows that this bijection is induced by an element of~$P\Sp(3)$.
Conversely, every~$P\Sp(3)$-image of~$X$ has~$315$ points and the
prescribed projective inner products.
\end{proof}

\begin{corollary}[geometric and combinatorial symmetries]\label{cor:projective-stabilizer}
For the Hall--Janko configuration,
\[
 \operatorname{Stab}_{P\Sp(3)}(X)\cong J_2,
 \qquad \Aut(\mathcal N_X)\cong J_2:2.
\]
The Galois permutation~$\iota$ represents the nontrivial outer
symmetry and is not induced by~$P\Sp(3)$.
\end{corollary}

\begin{proof}
The geometric stabilizer acts faithfully on~$X$ by
\cref{thm:quaternionic-cubic-rigidity}.  It contains the projective
image~$2.J_2/\{\pm1\}\cong J_2$ of the lattice group and embeds in
the full near-octagon automorphism group~$J_2:2$
\cite{DeWispelaereVanMaldeghem}.
The latter also contains~$\iota$, which exchanges two distinct real
inner products and therefore cannot belong to~$P\Sp(3)$.
The geometric stabilizer is consequently the index-two subgroup~$J_2$, and~$\iota$ lies in the outer coset.
\end{proof}

\begin{remark}[scope of the uniqueness statement]\label{rem:maximality-uniqueness}
The angle set is an explicit hypothesis in
\cref{thm:HJ-prescribed-angle-uniqueness}; the projective~$5$-design
property is forced by attaining the bound~$315$.  This does not prove
the converse implication from cardinality and design strength to the
angle set.  Under the angle hypothesis, equality in the bound,
Cohen--Tits uniqueness, the elementary fission argument, the internal
Galois symmetry, and cubic rigidity leave no combinatorial or
geometric ambiguity.
\end{remark}

\begin{problem}[unconditional angle rigidity]\label{prob:unconditional-angle-set}
Does every~$315$-point quaternionic projective~$5$-design in~$\Hh\PP^2$ necessarily have the five off-diagonal projective inner products in \eqref{eq:five-angle-set}?  An affirmative answer would imply that every~$315$-point quaternionic
projective~$5$-design is~$P\Sp(3)$-equivalent to~$X$.
\end{problem}

The same~$525$ frames link the arithmetic and geometry: their
projective moments determine the incidence structure, while one
framed reduction closes the mass calculation.  The cubic moment then
recovers the quaternionic Jordan product, completing the passage
from the arithmetic genus to the finite geometry and its rigid
projective realization.

\section*{Acknowledgments}
The author's interest in Hilbert modular surfaces and in the code--lattice correspondence goes back to Friedrich Hirzebruch, his advisor.  In particular, he learned the relation between self-dual quinary codes, cyclotomic lattices, and Hilbert modular forms in Hirzebruch's Bonn lecture course \emph{Kodierungstheorie und Beziehungen zur Geometrie} in the winter semester 1986/87, which he attended.  Hirzebruch's geometric interpretation of the Hilbert modular surface for~$\Q(\sqrt5)$ through the Clebsch--Klein cubic surface is particularly consonant with the icosahedral arithmetic used here.  The author is also grateful to the late Zvonimir Janko for their discussions about sporadic groups during the author's stay in Heidelberg in the academic year 2003/2004.

\end{document}